\documentclass{amsart}

\usepackage{amssymb,amsmath}
\usepackage[section]{algorithm}
\usepackage{algpseudocode}

\newtheorem{theorem}{Theorem}[section]
\newtheorem{proposition}[theorem]{Proposition}
\newtheorem{lemma}[theorem]{Lemma}
\newtheorem{corollary}[theorem]{Corollary}

\newtheorem{definition}[theorem]{Definition}

\newtheorem{remark}[theorem]{Remark}

\newtheorem{example}[theorem]{Example}

\def\S{{\mathbb S}}
\def\N{{\mathbb N}}

\def\KX{{K\langle X \rangle}}

\def\sgn{{\rm sgn}}
\def\lcm{{\mathrm{lcm}}}

\def\w{{\mathrm{w}}}

\def\Im{{\mathrm{Im\,}}}
\def\End{{\mathrm{End}}}

\def\Mon{{\mathrm{Mon}}}

\def\Gr{Gr\"obner}

\def\lm{{\mathrm{lm}}}
\def\lc{{\mathrm{lc}}}
\def\lt{{\mathrm{lt}}}
\def\LM{{\mathrm{LM}}}
\def\spoly{{\mathrm{spoly}}}

\def\Reduce{{\textsc{Reduce}}}
\def\SkewGBasis{{\textsc{SkewGBasis}}}

\def\SigmaGBasis{{\textsc{SigmaGBasis}}}
\def\FreeGBasis{{\textsc{FreeGBasis}}}

\def\hN{{\hat{\N}}}
\def\hSigma{{\hat{\Sigma}}}
\def\hS{{\hat{S}}}
\def\hpi{{\hat{\pi}}}

\begin{document}

\title[Skew polynomial rings, Gr\"obner bases $\ldots$]
{Skew polynomial rings, Gr\"obner bases and the letterplace
embedding of the free associative algebra}.

\author[R. La Scala]{Roberto La Scala$^*$}

\author[V. Levandovskyy]{Viktor Levandovskyy$^{**}$}

\address{$^*$ Dipartimento di Matematica, via Orabona 4, 70125 Bari, Italia}
\email{lascala@dm.uniba.it}

\address{$^{**}$ RWTH Aachen, Templergraben 64, 52062 Aachen, Germany}
\email{levandov@math.rwth-aachen.de}

\thanks{Partially supported by Universit\`a di Bari, Ministero dell'Universit\`a
e della Ricerca, and by the DFG Graduiertenkolleg ``Hierarchie und Symmetrie in
mathematischen Modellen'' at RWTH Aachen, Germany}




\begin{abstract}
In this paper we introduce an algebra embedding $\iota:\KX\to S$
from the free associative algebra $\KX$ generated by a finite or
countable set $X$ into the skew monoid ring $S = P * \Sigma$ defined
by the commutative polynomial ring $P = K[X\times\N^*]$ and by the monoid
$\Sigma = \langle \sigma \rangle$ generated by a suitable endomorphism
$\sigma:P\to P$. If $P = K[X]$ is any ring of polynomials in a countable
set of commuting variables, we present also a general \Gr\ bases theory
for graded two-sided ideals of the graded algebra $S = \bigoplus_i S_i$
with $S_i = P \sigma^i$ and $\sigma:P \to P$ an abstract endomorphism
satisfying compatibility conditions with ordering and divisibility
of the monomials of $P$. Moreover, using a suitable grading
for the algebra $P$ compatible with the action of $\Sigma$, we obtain
a bijective correspondence, preserving \Gr\ bases, between graded
$\Sigma$-invariant ideals of $P$ and a class of graded two-sided ideals
of $S$. By means of the embedding $\iota$ this results in the unification,
in the graded case, of the \Gr\ bases theories for commutative and
non-commutative polynomial rings. Finally, since the ring of ordinary
difference polynomials $P = K[X\times\N]$ fits the proposed theory
one obtains that, with respect to a suitable grading, the \Gr\ bases
of finitely generated graded ordinary difference ideals can be computed
also in the operators ring $S$ and in a finite number of steps up
to some fixed degree.
\end{abstract}

\maketitle

\section{Introduction}

Let $P$ be a $K$-algebra and let $\Sigma$ be a monoid of endomorphisms
of $P$. If $I$ is an ideal of $P$ which is invariant under the maps
in $\Sigma$ then it is possible to codify the action of $P$ and $\Sigma$
over $I$ as a single left module structure with respect to the skew monoid
(or semigroup) ring $S = P * \Sigma$. The study of some properties of $I$,
as for instance its finite $\Sigma$-generation, can be reduced hence
to that of general properties of the operators ring $S$ as its
Noetherianity (see \cite{MCR,La}). Ideals which are stable under the action
of monoids of endomorphisms or groups of automorphisms are natural
in many contexts as the representation theory (a classical reference
is \cite{DCEP}), or in the study of PI-algebras \cite{Dr,GZ} where
$P$ is the free associative algebra and $\Sigma$ the complete
monoid of endomorphisms of $P$. Another context of relevant interest
is the study of so-called ``difference ideals'' \cite{Le} which are
ideals invariant under shift operators in applications to combinatorics,
(nonlinear) differential and difference equations. For the viewpoint
of computing in the ring of (differential-difference) operators
an important contribution is \cite{MS}.

To control in an effective way the structure of the left $S$-module $P/I$
one generally needs to compute a $K$-basis of it. If $P$ is a ring
of polynomials in commutive or non-commutative variables and one fixes
a suitable ordering for the monomials of $P$, then a $K$-linear basis
of monomials for $P/I$ can be obtained by using the elements of a
suitable generating set of $I$ as rewriting rules. Such generating set
is usually called a ``\Gr\ basis'' of $I$. Since $I$ is a $\Sigma$-invariant
ideal, it is natural to consider $\Sigma$-bases of $I$ that is sets
$G\subset I$ such that $I$ is the smallest $\Sigma$-ideal of $P$
containing $G$. In other words, $G$ is a generating set of $I$ as left
$S$-module. It follows that one has to harmonize the notion of \Gr\ basis
with that of $\Sigma$-basis and attempts in this direction can be found
for instance in \cite{AH,BD} and also in \cite{DLS,LSL}. If the elements
of $\Sigma$ are automorphisms, the main obstacle in the definition
of a \Gr\ $\Sigma$-basis is that their action on $P$ does not preserve
the monomial ordering. Then, it has been shown in \cite{BD} and before
in \cite{LSL} that an appropriate setting to define \Gr\ $\Sigma$-bases
is that of a commutative polynomial ring $P = K[X]$ in an infinite
number of variables and a monoid $\Sigma$ of monomial monomorphisms
of infinite order which are compatible with the ordering and divisibility
of monomials of $P$.

In this paper we propose a systematic study of the case when $\Sigma$
is generated by a single map $\sigma$. In this case the skew monoid ring
$S$ coincides with the skew polynomial ring $P[s;\sigma]$ which is
an Ore extension where $\sigma$-derivation is zero. The approach
we follow is to consider an abstract map $\sigma$ satisfying compatibility
conditions able to provide a ``natural'' \Gr\ bases theory. Note that
this generalizes in particular the results contained in \cite{We}
where the map $\sigma:x_i\mapsto x_i^e$ with $e > 1$ has been studied.
We choose to consider a single endomorphism essentially because
a major application of our theory is the unification, in the graded case,
of the \Gr\ bases theory for non-commutative polynomials introduced
in \cite{Gr1,Mo,Uf} with the classical commutative theory based on
the notion of S-polynomial (see for instance \cite{GP}).
In Section 6 we show in fact that there exists an algebra embedding
$\iota:\KX\to S$ where $\KX$ is the free associative algebra generated
by the variables $x_i$ and $S$ is the skew polynomial ring defined by
the algebra $P$ of commutative polynomials in double indexed variables
$x_i(j)$ and the endomorphism $\sigma:P\to P$ such that
$x_i(j)\mapsto x_i(j+1)$, for all $i,j$. This algebra embedding is a
significant improvement of the linear map $\iota':\KX \to P$ defined as
$x_{i_1}\cdots x_{i_d}\mapsto x_{i_1}(1)\cdots x_{i_d}(d)$ and
introduced by \cite{Fe,DRS} for the aims of physics and
invariant/representation theory. In fact, the use of the map $\iota$
clarifies the phenomenon found in \cite{LSL} of a bijective correspondence
between all graded two-sided ideals of $\KX$ and some class of
$\Sigma$-invariant ideals of $P$. Note that in the same paper,
a competitive new algorithm for non-commutative homogeneous \Gr\ bases
based on this correspondence has been implemented and experimented
in \textsc{Singular} \cite{DGPS}.

In Section 2 one finds a brief account of the equivalence between
the notion of $\Sigma$-invariant $P$-module and that of left $S$-module,
together with the description of some properties of the generating sets
of graded two-sided ideals of $S = \bigoplus_i S_i$ with $S_i = P s^i$.
A \Gr\ basis theory for such ideals is introduced in Section 3 where
we assume $P = K[x_0,x_1,\ldots], \Sigma = \langle \sigma \rangle$ and
$\sigma:P\to P$ be a monomorphism of infinite order sending monomials into
monomials. Additional assuptions for $\sigma$ are that
$\gcd(\sigma(x_i),\sigma(x_j)) = 1$ for $i\neq j$ and the monomial ordering
of $P$ is such that $m\prec n$ implies that $\sigma(m)\prec \sigma(n)$,
for all monomials $m,n$. Such conditions are quite natural in many
contexts as the shift operators defining difference ideals \cite{Le}
or the maps used in \cite{BD}. Note that the algorithms we introduce
for the computation of homogeneous \Gr\ bases in $S$ are based on
the free $P$-module structure of this ring and hence they appear as
a variant of the classical module Buchberger algorithm where the number
of S-polynomials to be considered is reduced owing to the symmetry
defined by $\Sigma$ on the graded ideals of the ring $S$.

In Section 5 we analyze the notion of \Gr\ $\Sigma$-basis for
$\Sigma$-invariant ideals of $P$. When $P$ can be endowed
with a suitable grading compatible with the action of $\Sigma$,
we describe a bijective correspondence between all graded
$\Sigma$-invariant ideals of $P$ and some class of graded two-sided
ideals of $S$. Such correspondence preserves \Gr\ bases and gives
rise to a ``duality'' between homogeneous algorithms in $P$ and in $S$.
Note that for finitely generated ideals all these procedures admit
termination when truncated at some degree. As we said earlier, in Section 6
the algebra embedding $\iota:\KX\to S$ is introduced and a bijective
correspondence between the ideals of $\KX$ and suitable ideals of $S$
is hence obtained by extension. The \Gr\ bases are preserved by this
correspondence and one obtains an alternative algorithm for computing
non-commutative homogeneous \Gr\ bases of $\KX$ in the free $P$-module $S$.
By means of the bijection of the Section 5, we reobtain in Section 7
the ideal correspondence and related algorithms introduced in \cite{LSL}
which provide the computation of non-commutative homogeneous \Gr\ bases
directly in $P$. Therefore, the theory for such bases can be deduced
by the classical Buchberger algorithm for commutative polynomial rings.
In Section 8 we propose the explicit computation of a finite \Gr\ basis
of an ideal of ordinary difference polynomials that can be obtained
as a special case by the algorithms we introduced. Moreover, in this
section we provide some timings obtained by an improvement of the library
\texttt{freegb.lib} of \textsc{Singular} initially developed for
\cite{LSL}. Finally, in Section 9 we propose some conclusions and
suggestions for future developments of the theory of \Gr\ $\Sigma$-bases
and its methods.

The preliminary full-size version of this paper has been accepted for
oral presentation at MEGA 2011 conference in Stockholm.


\section{Modules over skew monoid rings}

Fix $K$ any field and let $P$ be a commutative $K$-algebra. Let now
$\Sigma\subset\End_K(P)$ a submonoid of the monoid of $K$-algebra
endomorphisms of $P$. Denote $S = P*\Sigma$ the {\em skew monoid ring
defined by $\Sigma$ over $P$} that is $S$ is the free $P$-module with
(left) basis $\Sigma$ and the multiplication is defined by the identity
$\sigma f = \sigma(f) \sigma$, for all $f\in P,\sigma\in\Sigma$.
If $\Sigma\neq \{id\}$ then $S$ is a non-commutative $K$-algebra where
the ring $P$ and the monoid $\Sigma$ are embedded. Note that if $\Sigma =
\langle \sigma \rangle$ with $\sigma:P\to P$ a map of infinite order
one has that $S\approx P[s;\sigma]$, the skew polynomial ring in the variable
$s$ and coefficients in $P$ defined by the endomorphism $\sigma$.
Moreover, if $P$ is a domain and all maps in $\Sigma$ are injective then
$S$ is also a domain. To simplify notations, we denote $f^\sigma = \sigma(f)$
for any $f\in P,\sigma\in\Sigma$.

\begin{definition}
Let $M$ be a $P$-module. We call $M$ {\em a $\Sigma$-invariant module}
if there is a monoid homomorphism $\rho:\Sigma\to End_K(M)$ such that
$\rho(\sigma)(f x) = f^\sigma\rho(\sigma)(x)$, for all $f\in P,x\in M$
and $\sigma\in\Sigma$. We denote as usual $\sigma\cdot x = \rho(\sigma)(x)$.
If $M,M'$ are $\Sigma$-invariant modules and $\varphi:M\to M'$ is a
$P$-module homomorphism such that $\varphi(\sigma\cdot x) =
\sigma\cdot\varphi(x)$ for all $x\in M,\sigma\in\Sigma$, then the map
$\varphi$ is called a {\em homomorphism of $\Sigma$-invariant modules}.
\end{definition}

\begin{proposition}
\label{invar2leftmod}
The category of $\Sigma$-invariant $P$-modules is equal to the category
of left $S$-modules.
\end{proposition}

\begin{proof}
Let $M$ be a left $S$-module. Then $M$ is a $P$-module since
$P\subset S$. By restriction to $\Sigma\subset S$, one has a monoid
homomorphism $\rho:\Sigma\to\End_K(M)$. Moreover we have
$\sigma\cdot(f x) = (\sigma f)\cdot x = (f^\sigma \sigma)\cdot x =
f^\sigma (\sigma\cdot x)$, for all $f\in P, x\in M$ and $\sigma\in\Sigma$.
Let now $M$ be a $\Sigma$-invariant $P$-module. We can define a left
$S$-module structure by putting $(\sum_i f_i \sigma_i)\cdot x =
\sum_i f_i (\sigma_i\cdot x)$ with $f_i\in P,\sigma_i\in\Sigma$ and
$x\in M$. Consider a homomorphism $\varphi:M\to M'$ of $\Sigma$-invariant
modules. Since $\varphi$ is $P$-linear, one has
$\varphi((\sum_i f_i \sigma_i)\cdot x) = \sum_i f_i \varphi(\sigma_i\cdot x) =
\sum_i f_i (\sigma_i\cdot \varphi(x)) = (\sum_i f_i \sigma_i)\cdot \varphi(x)$.
\end{proof}

Let $M$ be a left $S$-module and let $G = \{g_i\}\subset M$ be a generating
set of $M$. Note that $x\in M$ if and only if $x =
\sum_{i,\sigma} f_{i\sigma} \sigma\cdot g_i$ with $f_{i\sigma}\in P$
that is $M$ is generated by $\Sigma\cdot G = \{\sigma\cdot g_i\}_{i,\sigma}$
as $P$-module. We want now to describe homogeneous bases for graded
two-sided ideals of $S$. In fact, the algebra $S$ has a natural grading
over the monoid $\Sigma$ that is $S = \bigoplus_{\sigma\in\Sigma} S_\sigma$
and $S_\sigma S_\tau\subset S_{\sigma\tau}$ where $S_\sigma = P \sigma$.
Note that $S_{id} = P$, all $S_\sigma$ are $P$-submodules of $S$ and
$S_\sigma \tau = S_{\sigma\tau}$.

\begin{proposition}
\label{gen2lgen}
Let $J\subset S$ be a graded (two-sided) ideal and let $G\subset J$ be a set
of homogeneous elements. Then $G$ is a generating set of $J$ if and only if
$G\,\Sigma$ is a left basis of $J$ that is $\Sigma\,G\,\Sigma$
is a basis of $J$ as $P$-module.
\end{proposition}

\begin{proof}
Assume $G = \{g_i \sigma_i\}$ with $g_i\in P,\sigma_i\in\Sigma$, for all $i$.
Let $p_i,q_i\in S$ with $q_i = \sum_\sigma q_{i\sigma} \sigma$ and
$q_{i\sigma}\in P$. It is sufficient to note that $\sum_i p_i g_i \sigma_i q_i =
\sum_{i,\sigma} p_i g_i \sigma_i q_{i\sigma} \sigma =
\sum_{i,\sigma} p_i q_{i\sigma}^{\sigma_i} g_i \sigma_i\sigma$.
\end{proof}

\begin{corollary}
\label{poldiv2ldiv}
Let $f,g\in S$ and let $g$ be a homogeneous element. Then, one has that
$f = p g q$ with $p,q\in S$ if and only if $f$ belongs to the (graded)
left ideal generated by $\{g \sigma\}_{\sigma\in\Sigma}$.
\end{corollary}


\section{Monomial orderings and \Gr\ bases}

Denote $\N = \{0,1,\ldots\}$ the set of non-negative integers and let
$X = \{x_0,x_1,\ldots\}$ be a countable set. From now on,
{\em we make the assumption} that $P = K[X]$ is a commutative polynomial
ring in the variables set $X$. Starting from Section 6 we will assume
in particular that this set has the form $X\times\N$.
Moreover, we fix $\sigma:P\to P$ an algebra homomorphism of infinite
order and define the monoid $\Sigma = \langle \sigma \rangle \approx \N$.
Then, the skew monoid ring $S = P*\Sigma$ is isomorphic to the
skew polynomial ring $P[s;\sigma]$ and we identify $\Sigma = \{\sigma^i\}$
with the monoid $\{s^i\}$ of powers of the variable $s$. Note that $S$
is a free $P$-module of infinite rank. We denote $f^{s^i} = f^{\sigma^i} =
\sigma^i(f)$, for all $f\in P, i\geq 0$. Moreover, a homogeneous element
$f\in S_i = P s^i$ for some $i$ is also called {\em $s$-homogeneous}
and we put $\deg_s(f) = i$. Note finally that in the theory of difference
ideals \cite{Le}, the ring $S$ is called {\em ring of ordinary
difference operators over $P$}.

Denote by $\Mon(P)$ the set of all monomials of $P$ (including 1). Clearly,
$\Mon(P)$ is a multiplicative $K$-basis of $P$ that is $m n\in\Mon(P)$
for all $m,n\in\Mon(P)$. By definition of $S$, a $K$-basis of such algebra
is given by the elements $m s^i$ where $m\in\Mon(P)$ and $i\geq 0$ is an integer.
We call such elements the {\em monomials of $S$} and we denote the set of them
as $\Mon(S)$. Clearly $\Mon(P)\subset\Mon(S)$. Note that $\Mon(S)$ is
in fact the ``monomial basis'' of $S$ as a free $P$-module.

In what follows, {\em we assume} also that the endomorphism $\sigma:P\to P$
is injective and {\em monomial} that is it stabilizes the set $\Mon(P)$.
In other words, $\{\sigma(x_i)\}$ is a set of algebraically independent
monomials. Since $P$ is a domain, it follows that $S$ is also a domain
and the $K$-basis $\Mon(S)$ is multiplicative since $m s^i n s^j =
m n^{s_i} s^{i+j}\neq 0$, for all $m,n\in\Mon(P)$ and $i,j\geq 0$. 

We want to study now some divisibility relations in $\Mon(S)$.
Let $f,g\in S$. We say that {\em $f$ left-divides $g$} if there is
$a\in S$ such that $g = a f$. Clearly, left divisibility is a partial
ordering (up to units). Since $\sigma$ is a monomial injective map,
one has that if $f,g\in\Mon(S)$ then also $a\in\Mon(S)$.

\begin{proposition}
Let $v = m s^i, w = n s^j\in\Mon(S)$ with $m,n\in\Mon(P)$.
Then $v$ left-divides $w$ if and only if $i\leq j$ and $m^{s^{j-i}}\mid n$.
\end{proposition}

\begin{proof}
Let $a = p s^k\in\Mon(S)$ with $p\in\Mon(P)$ such that $n s^j = p s^k m s^i =
p m^{s^k} s^{k+i}$. Then, we have that $j - i = k\geq 0$ and $m^{s^k}\mid n$.
\end{proof}

Note that $S$ has also a free $P$-module structure and so $\Mon(S)$
inherits another notion of divisibility. Precisely, let $v,w\in\Mon(S)$.
We say that {\em $v$ $P$-divides $w$} if $\deg_s(v) = \deg_s(w)$ and 
there is $a\in\Mon(P)$ such that $w = a v$. Clearly $P$-divisibility
is a partial ordering and one has the following result.

\begin{proposition}
Let $v,w\in\Mon(S)$. Then $v$ left-divides $w$ if and only if
$s^k v$ $P$-divides $w$ for some $k\geq 0$.
\end{proposition}

Note that left divisibility coincides with $P$-divisibility when
the monomials have the same $s$-degree. If there are $v,w,a,b\in\Mon(S)$
such that $w = a v b$ we say that {\em $v$ (two-sided) divides $w$}.
It is easy to prove that such divisibility is also a partial ordering.

\begin{proposition}
\label{div2ldiv}
Let $v,w\in\Mon(S)$. Then $w$ is a multiple of $v$ if and only if
there is $j\geq 0$ such that $w$ is a left multiple of $v s^j$, that is
$w$ is a $P$-multiple of $s^i v s^j$ for some $i,j\geq 0$.
\end{proposition}

\begin{proof}
Since monomials are $s$-homogeneous elements of $S$, by applying
Corollary \ref{poldiv2ldiv} we obtain that $w$ is a multiple of $v$
if and only if $w$ belongs to the (graded) left ideal generated by
$\{v s^j\}_{j\geq 0}$. Clearly, this happens when $w$ is a left multiple
of $v s^j$ for some $j$.
\end{proof}

We start now considering monomial orderings.

\begin{definition}
Let $\prec$ be a total ordering on the set $\Mon(S)$. We call $\prec$
a {\em monomial ordering of $S$} if it satisfies the following conditions
\begin{itemize}
\item[(i)] $\prec$ is a well-ordering, that is every non-empty set
of $\Mon(S)$ has a minimal element;
\item[(ii)] $\prec$ is compatible with multiplication, that is
if $v\prec w$ then $p v q\prec p w q$, for all $v,w,p,q\in\Mon(S)$.
\end{itemize}
\end{definition}

It follows immediately that $1\preceq w$ for any $w\in\Mon(S)$
and if $w = p v q$ with $p\neq 1$ or $q\neq 1$ then $v\prec w$
for all $v,w,p,q\in\Mon(S)$. Note that the above conditions agree
with general definitions of orderings on $K$-bases of associative
algebras that provide a \Gr\ basis theory (see for instance \cite{Gr2,Li}).
The same conditions define monomial orderings of the free algebras $\KX$
and $K[X]$. Note that such algebras can be endowed with a monomial
ordering even if the set of variables $X$ is countable. This is
provided by the Higman's lemma \cite{Hi} which implies that any
multiplicatively compatible total ordering of the monomials
such that $1 \prec x_0 \prec x_1 \prec \ldots$ is a monomial ordering. 
Recall that $f^s$ stands for $\sigma(f)$ for any $f\in P$. 

\begin{definition}
Let $\prec$ be a monomial ordering on $P$. We call $\sigma$ {\em compatible
with $\prec$} if $\sigma$ is a strictly increasing map when restricted
to $\Mon(P)$, that is $m\prec n$ implies that $m^s\prec n^s$
for all $m,n\in\Mon(P)$.
\end{definition}

The following result is based essentially on Remark 3.2 in
\cite{BD}.

\begin{proposition}
\label{noauto}
Assume $\sigma$ be compatible with $\prec$. Then $\sigma$ is not an
automorphism and $m\preceq m^s$, for all $m\in\Mon(P)$.
\end{proposition}

\begin{proof}
Since $\sigma\neq id$, there is $m\in\Mon(S)$ such that $m\neq m^s$.
If $m\succ m^s$, by compatibility of $\sigma$ one gets an infinite
descending chain $m\succ m^s\succ m^{s^2}\succ\ldots$ which
contradicts the condition that $\prec$ is a well-ordering. We conclude
that $m\prec m^s$. Assume that $\sigma$ has the inverse $\sigma^{-1}$.
By applying $\sigma$, from $m^{s^{-1}}\prec n^{s^{-1}}$ it follows
that $m\prec n$. Since $\sigma^{-1}$ is injective, we have therefore that
$m\prec n$ implies that $m^{s^{-1}}\prec n^{s^{-1}}$. Now,
by compatibility of $\sigma^{-1}$ we obtain $m\prec m^{s^{-1}}$ which
contradicts $m\prec m^s$.
\end{proof}

There are many endomorphisms $\sigma$ with are compatible with
usual monomial orderings on $P$ like lex, degrevlex, etc. For instance,
we have the following maps.

\begin{itemize}
\item $\sigma(x_i) = x_{f(i)}$ for any $i$, where $f:\N\to\N$ is a
strictly increasing map. Such maps have been considered in \cite{BD}.
In particular, one may define the shift operator $\sigma(x_i) = x_{i+1}$
which is used in difference algebra.
\item $\sigma(x_i) = x_i^e$ for any $i$, with $e > 1$.
This map has been considered in \cite{We}.
\end{itemize}

\begin{proposition}
\label{toS2P}
Let $\prec$ be a monomial ordering on $S$. Then $\sigma$ is compatible with
the restriction of $\prec$ to $\Mon(P)$. Moreover, for any $m,n\in\Mon(P)$
and $i,j\geq 0$ one has that $m s^i\prec n s^j$ implies that
$m\prec n$ or $i < j$.
\end{proposition}

\begin{proof}
Suppose $m\prec n$ with $m,n\in\Mon(P)$. Then $s m\prec s n$ that is
$m^s s\prec n^s s$. If $m^s\succeq n^s$ then 
$m^s s\succeq n^s s$ which is a contradiction. We conclude
that $m^s\prec n^s$. Now, assume that $m\succeq n$ and $i\geq j$.
We have $m s^i\succeq m s^j\succeq n s^j$.
\end{proof}

Assume now $\sigma$ be compatible with a monomial ordering $\prec$ of $P$.
We define a total ordering on $\Mon(S)$ by putting $m s^i\prec' n s^j$
if and only if $i < j$, or $i = j$ and $m\prec n$, for all $m,n\in\Mon(P)$
and $i,j\geq 0$. Clearly, the restriction of $\prec'$ to $\Mon(P)$ is
$\prec$.

\begin{proposition}
\label{toP2S}
The ordering $\prec'$ is a monomial ordering on $S$ that extends $\prec$.
\end{proposition}

\begin{proof}
Clearly, an infinite descending sequence in $\Mon(S)$ implies 
an infinite descending sequence in $\Mon(P)$ which contradicts
the condition that $\prec$ is a well-ordering. Let $m s^i,n s^j\in\Mon(S)$
and suppose $m s^i\prec n s^j$ that is $i < j$, or $i = j$ and $m\prec n$.
Let $q s^k\in\Mon(S)$ and consider right multiplications $m s^i q s^k =
m q^{s^i} s^{i+k}$ and $n s^j q s^k = n q^{s^j} s^{j+k}$. If $i < j$ then 
$i + k < j + k$. If $i = j$ and $m\prec n$ then $m q^{s^i}\prec
n q^{s^i} = n q^{s^j}$. We conclude in both cases that
$m q^{s^i} s^{i+k} \prec n q^{s^j} s^{j+k}$. For left multiplications
$q s^k m s^i = q m^{s^k} s^{k+i}$ and $q s^k n s^j = q n^{s^k} s^{k+j}$,
note that $m\prec n$ implies that $m^{s^k}\prec n^{s^k}$. Then,
one may argue in a similar way as for right multiplications.
\end{proof}

Clearly, a byproduct of Proposition \ref{toS2P} and Proposition \ref{toP2S}
is that there exist monomial orderings on the skew polynomial ring $S$
if and only if $\sigma$ is compatible with a monomial ordering of $P$.
Note that $\prec'$ is well-known as module ordering when we consider
$S$ as a free $P$-module. Moreover, by Proposition \ref{toS2P} it follows
also that the monomial ordering of $S$ is uniquely defined by the one of $P$
when one compares monomials of the same $s$-degree.

From now on, {\em we assume} $S$ be endowed with a monomial ordering $\prec$.

\begin{definition}
Let $f\in S, f = \sum_i \alpha_i m_i s^i$ with $m_i\in\Mon(P),\alpha_i\in K^*$.
Then, we denote $\lm(f) = m_k s^k = \max_\prec\{m_i s^i\}$, $\lc(f) = \alpha_k$
and $\lt(f) = \lc(f)\lm(f)$. If $G\subset S$ we put $\lm(G) =
\{\lm(f) \mid f\in G,f\neq 0\}$. We denote as $\LM(G)$ and $\LM_l(G)$
respectively the two-sided ideal and the left ideal of $S$ generated
by $\lm(G)$. Moreover, we denote by $\LM_P(G)$ the $P$-submodule of $S$
generated by $\lm(G)$.
\end{definition}

\begin{proposition}
\label{normbas}
Let $J$ be an ideal (respectively left ideal) of $S$. Then, the set
$\{w + J \mid w\in\Mon(S)\setminus\LM(J)\}$
(resp.~$\{w + J \mid w\in\Mon(S)\setminus\LM_l(J)\}$) is a $K$-basis
of the space $S/J$. If $J\subset S$ is a $P$-submodule, in the same way
one defines the $K$-basis $\{w + J \mid w\in\Mon(S)\setminus\LM_P(J)\}$.
\end{proposition}

\begin{proof}
Let $w\in\Mon(S)$. By induction on the monomial ordering of $S$, we can assume
that for any monomial $v\in\Mon(S)$ such that $v\prec w$ there is
a polynomial $f\in S$ belonging to the span of $N = \Mon(S)\setminus\LM(J)$
such that $v - f\in J$. If $w\notin N$ then there is $g\in J$ such that
$w = p \lm(g) q$ with $p,q\in\Mon(S)$. Therefore $f = w - (1/\lc(g))p g q$
is such that $\lm(f)\prec w$ and by induction $f - f'\in J$ for some
$f'\in\langle N \rangle_K$. We conclude that $w - f'\in J$. Finally
if $f\in N\cap J$ then necessarily $f = 0$.
Mutatis mutandis one proves the remaining assertions.
\end{proof}

\begin{definition}
Let $J$ be an ideal (respectively left ideal) of $S$ and $G\subset J$.
We call $G$ a {\em \Gr\ basis (resp.~left basis)} of $J$ if
$\LM(G) = \LM(J)$ (resp.~$\LM_l(G) = \LM_l(J)$). As usual, if $J$
is a $P$-submodule of $S$ then $G$ is a {\em \Gr\ $P$-basis} of $J$
when $\LM_P(G) = \LM_P(J)$.
\end{definition}

\begin{proposition}
Let $J$ be an ideal (respectively left ideal) of $S$ and $G\subset J$. 
The following conditions are equivalent:
\begin{itemize}
\item[(i)] $G$ is a \Gr\ basis (resp.~left basis) of $J$;
\item[(ii)] for any $f\in J$, one has a {\em \Gr\ representation of $f$
with respect to $G$} that is $f = \sum_i f_i g_i h_i$
(resp.~$f = \sum_i f_i g_i$) with $\lm(f)\succeq\lm(f_i)\lm(g_i)\lm(h_i)$
(resp.~$\lm(f)\succeq\lm(f_i)\lm(g_i)$) and $f_i,h_i\in S$, for all $i$.
\end{itemize}
A similar characterization holds for \Gr\ $P$-bases.
\end{proposition}

\begin{proof}
It follows immediately by the reduction process which is implicit
in the proof of Proposition \ref{normbas}.
\end{proof}

\begin{proposition}
\label{gb2lgb}
Let $J$ be a graded ideal of $S$ and $G\subset J$ be a subset of
$s$-homogeneous elements. The following conditions are equivalent:
\begin{itemize}
\item[(i)] $G$ is a \Gr\ basis of $J$;
\item[(ii)] $G\,\Sigma$ is a \Gr\ left basis of $J$;
\item[(iii)] $\Sigma\,G\,\Sigma$ is a \Gr\ $P$-basis of $J$.
\end{itemize}
\end{proposition}

\begin{proof}
Assume $G = \{g_i\}$ is a \Gr\ basis of $J$ and put $d_i = \deg_s(g_i)$.
If $f\in J$ then one has $f = \sum_i f_i g_i h_i$ where $f_i,h_i\in S$
and $\lm(f)\succeq\lm(f_i)\lm(g_i)\lm(h_i)$, for all $i$. Decompose
$h_i = \sum_j h_{i j} s^j$ with $h_{i j}\in P$ for any $i,j$.
Then, we have $\lm(f)\succeq\lm(f_i)\lm(g_i)\lm(h_{i j}) s^j$, for all $i,j$.
Since $\lm(g_i)$ has $s$-degree $d_i$, one obtains
$\lm(f_i)\lm(g_i)\lm(h_{i j}) s^j =
\lm(f_i)\lm(h_{i j})^{s^{d_i}}\lm(g_i s^j)$.
Moreover, as in Proposition \ref{gen2lgen}, we have
$f = \sum_{i,j} f_i g_i h_{i j} s^j =
\sum_{i,j} f_i h_{i j}^{s^{d_i}} g_i s^j$.
From $\sigma$ compatible with $\prec$ it follows that
$\lm(h_{i j}^{s^{d_i}}) = \lm(h_{i j})^{s^{d_i}}$ and hence $f$ has a left
\Gr\ representation with respect to $G\,\Sigma$, that is this set
is a left \Gr\ basis of $J$. The rest of the proof is straightforward.
\end{proof}


\section{Buchberger algorithm}

After Proposition \ref{gb2lgb}, in order to obtain a homogeneous \Gr\ basis
$G$ of a (two-sided) graded ideal $J\subset S$ one has to start with a
homogeneous generating set $H$ and consider the $P$-basis
$H' = \Sigma\,H\,\Sigma$. Then, one should transform $H'$ into a
homogeneous \Gr\ $P$-basis $G'$ of $J$ and finally reduce $G'$
as $G' = \Sigma\,G\,\Sigma$ with $G\subset J$.
Apart with problems concerning termination of the module Buchberger
algorithm ($P$ is not Noetherian and $S$ is a $P$-module of countable
rank) that we will show solvable for the truncated algorithm up to
some $s$-degree (see Proposition \ref{termin}), it is more desirable
to have a procedure able to compute $G$ without actually considering
all elements of $G'$. To obtain this, we need an additional requirement
for the endomorphism $\sigma$.

Note that, since $\sigma:P\to P$ is a ring homomorphism, such map
is increasing with respect to the divisibility relation in $P$,
that is $f\mid g$ implies that $f^s\mid g^s$ and in this case $(g/f)^s =
g^s/f^s$ with $f,g\in P$.

\begin{proposition}
\label{divcompat}
The following conditions are equivalent:
\begin{itemize}
\item[(a)] $\gcd(x_i^s,x_j^s) = 1$, for all $i\neq j$;
\item[(b)] $\gcd(m^s,n^s) = \gcd(m,n)^s$, for all
$m,n\in\Mon(P)$.
\end{itemize}
Moreover, in this case one has $m\mid n$ if and only if $m^s\mid n^s$
and $\lcm(m^s,n^s) = \lcm(m,n)^s$ with $m,n\in P$.
In other words, $\sigma$ is a lattice homomorphism with respect to the
divisibility relation in $\Mon(P)$.
\end{proposition}

\begin{proof}
Assume (a) and let $m,n\in\Mon(P)$ such that $\gcd(m,n) = 1$.
If $m = x_{i_1}\cdots x_{i_k}$ and $n = x_{j_1}\cdots x_{j_l}$ then
$m^s = x_{i_1}^s\cdots x_{i_k}^s$ and $n^s = x_{j_1}^s\cdots x_{j_l}^s$
with $\{i_1,\ldots,i_k\}\cap\{j_1,\ldots,j_l\} = \emptyset$.
Since $\gcd(x_i^s,x_j^s) = 1$ for all $i\neq j$, we conclude that
$\gcd(m^s,n^s) = 1$. Assume now $\gcd(m,n) = u$ and hence
$\gcd(m/u,n/u) = 1$. Then $\gcd(m^s/u^s,n^s/u^s) =
\gcd((m/u)^s,(n/u)^s) = 1$ and therefore $\gcd(m^s,n^s) = u^s$
that is (b) holds. Suppose $m^s\mid n^s$ that is $m^s = \gcd(m^s,n^s) =
\gcd(m,n)^s$. Since $\sigma$ is injective we have that $m = \gcd(m,n)$
that is $m\mid n$. Moreover, one obtains $\lcm(m,n)^s = (m n / \gcd(m,n))^s =
(m n)^s / \gcd(m,n)^s = m^s n^s / \gcd(m^s,n^s) = \lcm(m^s,n^s)$
for all $m,n\in\Mon(P)$.
\end{proof}

\begin{definition}
We say that {\em $\sigma$ is compatible with divisibility in $\Mon(P)$}
if for all $i\neq j$, one has $\gcd(x_i^s,x_j^s) = 1$ that is the variables
occuring in the monomials $x_i^s,x_j^s$ are disjoint.
\end{definition}

Note that if a monomial endomorphism of $P$ is compatible with divisibility
then it is automatically injective since the monomials $x_i^s$ are
algebraically independent.
Let $\mid$ be the divisibility relation and $\prec$ a monomial ordering
on $\Mon(P)$. Throughout the rest of the paper, {\em we make the assumption}
that the monomial endomorphism $\sigma:P \to P$ is compatible both with
$\mid$ and with $\prec$.

We recall now some basic results in the theory of module \Gr\ bases
by applying them to the free $P$-module $S$ whose (left) free basis is
$\Sigma = \{s^i\}_{i\geq 0}$. Consider $f,g\in S\setminus\{0\}$
two elements whose leading monomials have the same $s$-degree (component),
that is $\lm(f) = m s^i, \lm(g) = n s^i$ with $m,n\in\Mon(P)$ and $i\geq 0$.
If we put $\lc(f) = \alpha, \lc(g) = \beta$ and $l = \lcm(m,n)$,
one defines the \mbox{\em S-polynomial} $\spoly(f,g) =
(l/\alpha m) f - (l/\beta n) g$. Clearly $\spoly(f,g) = - \spoly(g,f)$
and $\spoly(f,f) = 0$.

\begin{proposition}[Buchberger criterion]
Let $G$ be a generating set of a $P$-submo\-du\-le $J\subset S$. Then
$G$ is a \Gr\ basis of $J$ if and only if for all $f,g\in G\setminus\{0\}$
such that $\deg_s(\lm(f)) = \deg_s(\lm(g))$, the S-polynomial $\spoly(f,g)$
has a \Gr\ representation with respect to $G$.
\end{proposition}

Usually the above result, see for instance \cite{Ei,GP},
is stated when $P$ is a polynomial ring with a finite number of variables
and $S$ is a $P$-module of finite rank. In fact such assumptions are
not needed since Noetherianity is not used in the proof, but only the
existence of a monomial ordering for the ring $P$ and the free module $S$.
See also the comprehensive Bergman's paper \cite{Be} where the ``Diamond
Lemma'' is proved without any restriction on the finiteness of the variable
set. In the following results we show how the Buchberger criterion, and
hence the corresponding algorithm, can be reduced up to the symmetry
defined by the monoid $\Sigma$ on $S$.

\begin{lemma}
\label{sigmaspoly}
Let $f,g\in S\setminus\{0\}$ and let $i\leq j$ such that
$\deg_s(lm(f)) + i = \deg_s(\lm(g)) + j$. Then
$\spoly(s^i f, s^j g) = s^i \spoly(f, s^{j-i} g)$ and $\spoly(f s^i, g s^j) =
\spoly(f, g s^{j-i}) s^i$.
\end{lemma}

\begin{proof}
Let $\lt(f) = \alpha m s^k, \lt(g) = \beta n s^l$ with
$\alpha,\beta\in K^*$ and $m,n\in\Mon(P)$. Then $\lt(s^i f) = 
\alpha m^{s^i} s^{i+k}, \lt(s^j g) = \beta n^{s^j} s^{j+l}$
and $\lt(s^{j-i} g) = \beta n^{s^{j-i}} s^{j-i+l}$.
By compatibility of $\sigma$ with divisibility in $\Mon(P)$,
if $q = \lcm(m,n^{s^{j-i}})$ then $\lcm(m^{s^i}, n^{s^j}) = q^{s^i}$.
Therefore $h = \spoly(f, s^{j-i} g) =
(q/\alpha m) f - (q/\beta n^{s^{j-i}}) s^{j-i} g$ and hence we have
$s^i h = (q^{s^i}/\alpha m^{s^i}) s^i f - (q^{s^i}/\beta n^{s^j}) s^j g =
\spoly(s^i f, s^j g)$.

Note now that $\lt(f s^i) =  \alpha m s^{i+k}, \lt(g s^j) = \beta n s^{j+l}$
and $\lt(g s^{j-i}) = \beta n s^{j-i+l}$. If $q = \lcm(m,n)$ and
$h = \spoly(f, g s^{j-i}) = (q/\alpha m) f - (q/\beta n) g s^{j-i}$ we have
simply that $h s^i = (q/\alpha m) f s^i - (q/\beta n) g s^j =
\spoly(f s^i, g s^j)$.
\end{proof}

\begin{proposition}[Two-sided $\Sigma$-criterion]
\label{sigmacrit}
Let $G$ be an $s$-homogeneous basis of a graded two-sided ideal $J\subset S$.
Then $G$ is a \Gr\ basis of $J$ if and only if for all $f,g\in G\setminus\{0\}$
and for any $i,j\geq 0$, the S-polynomials $\spoly(f, s^i g s^j)$
$(\deg_s(f) = \deg_s(g) + i + j)$ and $\spoly(f s^i, s^j g)$ $(\deg_s(f) + i =
\deg_s(g) + j)$ have a \Gr\ representation with respect to
$\Sigma\,G\,\Sigma$.
\end{proposition}

\begin{proof}
By Proposition \ref{gb2lgb} we have to prove that $G' = \Sigma\,G\,\Sigma$
is a \Gr\ basis of $J$ as $P$-module, that is $G'$ is $P$-basis of $J$ and
the S-polynomials $h = \spoly(s^i f s^k, s^j g s^l)$ have a \Gr\ representation
with respect to $G'$ for all $f,g\in G\setminus\{0\}$ and for any $i,j,k,l\geq 0$
such that $\deg_s(f) + i + k = \deg_s(g) + j + l$. Since $G$ is a homogeneous
basis of $J$ as two-sided ideal, from Proposition \ref{gen2lgen} it follows that
$G'$ is a generating set of $J'$ as $P$-module.
Consider now all possibilities $i\leq j$ or $i\geq j$ and $k\leq l$ or $k\geq l$
and apply Lemma \ref{sigmaspoly}.
If $i\leq j,k\leq l$ one has $h = s^i \spoly(f, s^{j-i} g s^{l-k}) s^k$,
if $i\leq j,k\geq l$ then $h = s^i \spoly(f s^{l-k}, s^{j-i} g) s^l$, and
so on. Then, assume that a S-polynomial $h = \spoly(f,g)$, with
$f,g\in G'\setminus\{0\}$, has a \Gr\ representation with respect to $G'$
as $P$-basis of $J$, that is $h = \sum_i f_i g_i$ with $f_i\in P, g_i\in G'$
and $\lm(h)\geq\lm(f_i)\lm(g_i)$, for all $i$. We have to prove that
$s^k h s^l$ has also a \Gr\ representation with respect to $G'$
for any $k,l\geq 0$. One has that $s^k h s^l = \sum_i f_i^{s^k} s^k g_i s^l$
and $\lm(s^k h s^l) = s^k \lm(h) s^l\geq s^k \lm(f_i)\lm(g_i) s^l =
\lm(f_i)^{s^k} s^k \lm(g_i) s^l = \lm(f_i^{s^k}) \lm(s^k g_i s^l)$.
Since $s^k g_i s^l\in G' = \Sigma\,G\,\Sigma$, one obtains that
$s^k h s^l$ has a \Gr\ representation with respect to $G'$.
\end{proof}

A criterion similar to Proposition \ref{sigmacrit} holds clearly for
\Gr\ left bases of left ideals of $S$ where no restrictions about
the $s$-homogeneity of bases and ideals are needed.

\begin{proposition}[Left $\Sigma$-criterion]
\label{leftsigmacrit}
Let $G$ be a basis of a left ideal $J\subset S$. Then $G$ is a \Gr\ basis
of $J$ if and only if for all elements $f,g\in G\setminus\{0\}$ such that
$i = \deg_s(\lm(f)) - \deg_s(\lm(g))\geq 0$, the S-polynomial
$\spoly(f, s^i g)$ has a \Gr\ representation with respect to $\Sigma\,G$.
\end{proposition}

A standard procedure in the (module) Buchberger algorithm is the following.

\suppressfloats[b]
\begin{algorithm}\caption{\Reduce}
\begin{algorithmic}[0]
\State \text{Input:} $f\in S$ and $G\subset S$.
\State \text{Output:} $h\in S$ such that $f - h\in\langle G\rangle_P$
and $h = 0$ or $\lm(h)\notin\LM_P(G)$.
\State $h:= f$;
\While{ $h\neq 0$ and $\lm(h)\in\LM_P(G)$ }
\State choose $g\in G,g\neq 0$ such that $\lm(g)$ $P$-divides $\lm(h)$;
\State $h:= h - (\lt(h)/\lt(g)) g$;
\EndWhile;
\State \Return $h$.
\end{algorithmic}
\end{algorithm}


Note that if $\lt(g) = \alpha m s^i, \lt(h) = \beta n s^i$ with
$\alpha,\beta\in K^*$ and $m,n\in\Mon(P)$, by definition
$\lt(h)/\lt(g) = (\alpha m)/(\beta n)$. Moreover, the termination
of $\Reduce$ is provided since $\prec$ is a well-ordering on $\Mon(S)$.
In particular, even if $G$ is an infinite set, there are only a finite
number of elements $g\in G,g\neq 0$ such that $\lm(g)$ $P$-divides $\lm(h)$
and hence $\lm(g)\preceq\lm(h)$.

It is well-known that if $\Reduce(f,G) = 0$ then $f$ has a \Gr\ representation
with respect to $G$. Moreover, if $\Reduce(f,G) = h\neq 0$ then clearly
we have $\Reduce(f,G\cup\{h\}) = 0$. Therefore, from Proposition \ref{sigmacrit}
it follows immediately the correctness of the following algorithm.

\suppressfloats[b]
\begin{algorithm}\caption{\SkewGBasis}
\begin{algorithmic}[0]
\State \text{Input:} $H$, an $s$-homogeneous basis of a graded two-sided
ideal $J\subset S$.
\State \text{Output:} $G$, an $s$-homogeneous \Gr\ basis of $J$.
\State $G:= H$;
\State $B:= \{(f,g) \mid f,g\in G\}$;
\While{$B\neq\emptyset$}
\State choose $(f,g)\in B$;
\State $B:= B\setminus \{(f,g)\}$;
\ForAll{$i,j\geq 0$ such that $i + j = \deg_s(f) - \deg_s(g)$}
\State $h:= \Reduce(\spoly(f,s^i g s^j), \Sigma\,G\,\Sigma)$;
\If{$h\neq 0$}
\State $B:= B\cup\{(h,h),(h,k),(k,h)\mid k\in G\}$;
\State $G:= G\cup\{h\}$;
\EndIf;
\EndFor;
\ForAll{$i,j\geq 0$ such that $j - i = \deg_s(f) - \deg_s(g)$}
\State $h:= \Reduce(\spoly(f s^i, s^j g), \Sigma\,G\,\Sigma)$;
\If{$h\neq 0$}
\State $B:= B\cup\{(h,h),(h,k),(k,h)\mid k\in G\}$;
\State $G:= G\cup\{h\}$;
\EndIf;
\EndFor;
\EndWhile;
\State \Return $G$.
\end{algorithmic}
\end{algorithm}

\newpage

Clearly, all well-known criteria (product criterion, chain criterion, etc)
can be applied to \SkewGBasis\ to shorten the number of S-polynomials
to be considered. In fact, this algorithm can be understood as the usual
(module) Buchberger procedure applied to the $P$-basis $\Sigma\,H\,\Sigma$,
where an additional criterion to avoid ``useless pairs'' is provided
by Proposition \ref{sigmacrit}. Note that owing to Proposition
\ref{leftsigmacrit}, one has also a similar procedure for computing a
\Gr\ left basis of any left ideal of $S$. Since the set $\Sigma\,H\,\Sigma$
if infinite even if the basis $H$ is eventually finite ($S$ is a
non-Noetherian ring) one has that \SkewGBasis\ does not admit general
termination. In particular, the cycle ``{\bf for all} $i,j\geq 0$
such that $j - i = \deg_s(f) - \deg_s(g)$ {\bf do}'' never stops unless
one bounds the $s$-degree $\deg_s(f) + i = \deg_s(g) + j$.
As for other non-Noetherian structures like the free associative
algebra that in fact can be embedded in $S$ (see Section 6),
the termination of homogeneous \Gr\ bases computations can be obtained
only under truncation.

\begin{proposition}
\label{termin}
Let $J\subset S$ be a graded two-sided ideal and fix $d\geq 0$.
Assume that $J$ has a $s$-homogeneous basis $H$ such that $H_d =
\{f\in H\mid \deg_s(f)\leq d\}$ is a finite set. Then, there exists
an $s$-homogeneous \Gr\ basis $G$ of $J$ such that $G_d$ is also finite.
In other words, if we consider a selection strategy for the S-polynomials
based on their $s$-degree, we obtain that the $d$-truncated version
of the algorithm \SkewGBasis\ terminates in a finite number of steps.
\end{proposition}

\begin{proof}
Denote $H'_d = \{ s^i f s^j\mid f\in H_d, i,j\geq 0, i + j + \deg_s(f)\leq d \}$.
Since $H_d$ is finite one has that $H'_d$ is also finite. Then, consider $X_d$
the finite set of variables of $P$ occurring in the elements of $H'_d$
and define $P^{(d)} = K[X_d]$ and $S^{(d)} = \bigoplus_{i\leq d} P^{(d)} s^i$.
In fact, the $d$-truncated algorithm \SkewGBasis\ computes a subset
of a \Gr\ basis of the $P^{(d)}$-submodule $J^{(d)}\subset S^{(d)}$
generated by $H'_d$. By Noetherianity of the ring $P^{(d)}$ and the free
$P^{(d)}$-module $S^{(d)}$ which has finite rank, we clearly obtain termination.
\end{proof}

Note that the above result implies algorithmic solution of the membership
problem for graded ideals of $S$ which are finitely generated up to
any degree.


\section{$\Sigma$-invariant ideals of $P$}

In this section we define \Gr\ bases of $\Sigma$-invariant
ideals $I\subset P$ which generates $I$ up to the action of $\Sigma$.
Moreover, if $P$ can be endowed with a suitable grading, we show how
such bases can be computed in the algebra $S$ for a class of graded
$\Sigma$-invariant ideals. As usual, we fix a monomial endomorphism
$\sigma:P\to P$ which is compatible both with the divisibility and
a monomial ordering on $\Mon(P)$ and we extend this to an ordering
on $\Mon(S)$. From Section 2 we know that $\Sigma$-invariant ideals
of $P$ are just left $S$-submodules of $P$. Since we make use of
identification $\Sigma = \{s^i\}$, for all $f\in P\subset S$ and
for any $i\geq 0$ one has that $s^i\cdot f = f^{s^i} =
\sigma^i(f)$ and $s^i f = (s^i\cdot f) s^i$.

\begin{definition}
Let $I\subset P$ be a $\Sigma$-invariant ideal and $G\subset I$.
We say that $G$ is a {\em $\Sigma$-basis} of $I$ if $G$ is a basis of $I$
as left $S$-module. In other words, $\Sigma\cdot G$ is a basis of $I$
as $P$-ideal.
\end{definition}

\begin{proposition}
Let $G\subset P$. Then $\lm(\Sigma\cdot G) = \Sigma\cdot \lm(G)$.
In particular, if $I$ is a $\Sigma$-invariant $P$-ideal then $\LM_P(I)$
is also $\Sigma$-invariant.
\end{proposition}

\begin{proof}
Since $\sigma$ is compatible with the monomial ordering of $P$, it is sufficient
to note that $\lm(s^i\cdot f) = s^i\cdot\lm(f)$ for any $f\in P$ and $i\geq 0$.
\end{proof}

\begin{definition}
Let $I\subset P$ be a $\Sigma$-invariant ideal and $G\subset I$.
We call $G$ a {\em \Gr\ $\Sigma$-basis} of $I$ if $\lm(G)$ is a basis
of $\LM_P(I)$ as left $S$-module. In other words, $\Sigma\cdot G$
is a \Gr\ basis of $I$ as $P$-ideal.
\end{definition}

The computation of \Gr\ $\Sigma$-bases of $\Sigma$-invariant $P$-ideals
is relevant, for instance, in applications to difference algebra
(cf.~\cite{Le}, Chapter 3). Such computations appear also
in other contexts, see for instance \cite{DLS} and \cite{BD}. Note that
in the latter paper \Gr\ $\Sigma$-bases are named ``equivariant
\Gr\ bases''.

In analogy with Proposition \ref{sigmacrit} and Proposition
\ref{leftsigmacrit}, we present here a $\Sigma$-criterion that allows
to reduce the number of S-polynomials to be checked to provide that
a $\Sigma$-basis is of \Gr\ type.

\begin{proposition}[$\Sigma$-criterion in $P$]
\label{Psigmacrit}
Let $G$ be a $\Sigma$-basis of a $\Sigma$-invariant ideal $I\subset P$.
Then $G$ is a \Gr\ $\Sigma$-basis of $I$ if and only if for all
$f,g\in G\setminus\{0\}$ and for any $i\geq 0$, the S-polynomial
$\spoly(f, s^i\cdot g)$ has a \Gr\ representation with respect to
$\Sigma\cdot G$.
\end{proposition}

\begin{proof}
Consider any pair of elements $s^i\cdot f, s^j\cdot g\in \Sigma\cdot G$
($f,g\in G$) and let $i\leq j$. By compatibility of $\sigma$ with
divisibility in $\Mon(P)$ (cf.~Lemma \ref{sigmaspoly}), one has that
$\spoly(s^i\cdot f, s^j\cdot g) = s^i\cdot \spoly(f, s^k\cdot g)$
with $k = j - i$. Assume that $\spoly(f, s^k\cdot g) = h =
\sum_l f_l (s^l\cdot g_l)$ ($f_l\in P, g_l\in G$) is a \Gr\
representation with respect to $\Sigma\cdot G$. Since the endomorphism
$\sigma$ is compatible with the monomial ordering of $P$, we have also
the \Gr\ representation $\spoly(s^i\cdot f, s^j\cdot g) = 
s^i\cdot h = \sum_l (s^i\cdot f_l) (s^{i+l}\cdot g_l)$.
\end{proof}

Note that some version of this criterion can be found in \cite{BD},
Theorem 2.5, where it is called ``equivariant Buchberger criterion''.
Before than this, the same ideas have been used in \cite{LSL}
for the Proposition 3.11. From this criterion it follows immediately
the correctness of the following algorithm.

\suppressfloats[b]
\begin{algorithm}\caption{SigmaGBasis}
\begin{algorithmic}[0]
\State \text{Input:} $H$, a $\Sigma$-basis of a $\Sigma$-invariant ideal
$I\subset P$.
\State \text{Output:} $G$, a \Gr\ $\Sigma$-basis of $I$.
\State $G:= H$;
\State $B:= \{(f,g) \mid f,g\in G\}$;
\While{$B\neq\emptyset$}
\State choose $(f,g)\in B$;
\State $B:= B\setminus \{(f,g)\}$;
\ForAll{$i\geq 0$}
\State $h:= \Reduce(\spoly(f,s^i\cdot g), \Sigma\cdot G)$;
\If{$h\neq 0$}
\State $B:= B\cup\{(h,h),(h,k),(k,h) \mid k\in G\}$;
\State $G:= G\cup\{h\}$;
\EndIf;
\EndFor;
\EndWhile;
\State \Return $G$.
\end{algorithmic}
\end{algorithm}

\newpage

As for the algorithm \SkewGBasis, all criteria to avoid useless pairs
can be added to \SigmaGBasis. Note that termination of this algorithm
is not provided in general (note the infinite cycle ``{\bf for all}
$i\geq 0$ {\bf do}'') and this is, in fact, one of the main problems
in applications to differential/difference algebra. Nevertheless,
in what follows we describe some class of $\Sigma$-invariant ideals
of $P$ where a truncated version of the algorithm \SigmaGBasis\ stops
in a finite number of steps. Such ideals are in bijective correspondence
with a class of graded (two-sided) ideals of $S$ which have truncated
termination of \SkewGBasis\ provided by Proposition \ref{termin}.

Consider now the $P$-module homomorphism $\pi:S\to P$ such that
$s^i\mapsto 1$, for all $i$. Clearly $\pi$ is a left $S$-module
epimorphism whose kernel is the left ideal of $S$ generated by $s - 1$.

\begin{definition}
Let $J$ be a graded ideal of $S$ and put $J^P = \pi(J)$. Clearly
$J^P$ is a $\Sigma$-invariant ideal of $P$.
\end{definition}

\begin{proposition}
\label{genS2P}
Let $J\subset S$ be a graded ideal. If $G$ is a homogeneous basis of $J$
then $G^P = \pi(G)$ is a $\Sigma$-basis of $J^P$.
\end{proposition}

\begin{proof}
Since the map $\pi$ is a left $S$-module homomorphism,
it is sufficient to note that $G\,\Sigma$ is a left basis of $J$ and
$\pi(G\,\Sigma) = \pi(G) = G^P$.
\end{proof}


We introduce now a grading on the algebra $P$ which is compatible
with action of $\Sigma$. We start extending the structure $(\N,\max,+)$
in the following way.

\begin{definition}
Let $-\infty$ be an element disjoint by $\N$ and put $\hN =
\{-\infty\}\cup\N$. Then, we define a commutative idempotent monoid
$(\hN,\max)$ with identity $-\infty$ that extends $(\N,\max)$ (with
identity 0) by imposing that $\max(-\infty,x) = x$ for any $x\in\hN$.
Moreover, we define a commutative monoid $(\hN,+)$ with identity $-\infty$
extending the monoid $(\N,+)$ by putting $-\infty + x = -\infty$,
for all $x\in\hN$. Since $+$ clearly distributes over $\max$, one has
that $(\hN,\max,+)$ is a commutative idempotent semiring, also known
as commutative dioid or max-plus algebra \cite{GM}.
\end{definition}

Note that if $\sigma^{-\infty}\in\End_K(P)$ is the map such that
$x_i\mapsto 0$ for any $x_i\in X$, then $\hSigma =
\{\sigma^{-\infty}\}\cup\Sigma$ is a commutative monoid isomorphic to
$(\hN,+)$. Denote now $M = \Mon(P)$ the set of monomials of the polynomial
ring $P$. 

\begin{definition}
\label{weight}
A mapping $\w:M\to\hN$ such that for all $m,n\in M$ and $x_i\in X$ one has
\begin{itemize}
\item[(i)] $\w(1) = -\infty$;
\item[(ii)] $\w(m n) = \max(\w(m),\w(n))$;
\item[(iii)] $\w(s\cdot x_i) = 1 + \w(x_i)$.
\end{itemize}
is called a {\em weight function of $P$ endowed with $\sigma$}.
\end{definition}

Note that if $m = x_{i_1}\cdots x_{i_d}\neq 1$ with
$\w(x_{i_1})\leq\ldots\leq\w(x_{i_d})$ then $\w(m) = \w(x_{i_d})$.
Moreover, the condition (iii) implies that $\w(s^i\cdot m) = i + \w(m)$
for all $i\in\hN, m\in M$ and hence $s^i\cdot m = m$ if and only if
$m = 1$ or $i = 0$. We put $M_i = \{m\in M\mid\w(m) = i\}$ for all
$i\in\hN$ and define $P_i\subset P$ the subspace spanned by $M_i$. 
We have that $P_{-\infty} = K$. Clearly $P = \bigoplus_{i\in\hN} P_i$
is a grading of the algebra $P$ defined by the monoid $(\hN,\max)$
by means of the function $\w$. Then, an element $f\in P_i$ is said
{\em $\w$-homogeneous of weight $i$}.

In what follows, {\em we assume} that $P$ is endowed with a weight function.
In fact, such functions are easily to define. Consider for instance
the polynomial ring $P = K[X\times\N]$ and denote $x_i(j)$ each variable
$(x_i,j)\in X\times\N$. Let $\sigma:P\to P$ be the algebra monomorphism
of infinite order such that $\sigma(x_i(j)) = x_i(j+1)$, for all $i,j$.
Clearly $\sigma$ is a monomial map compatible with divisibility in $\Mon(P)$
and many usual monomial orderings on $P$, like lex, degrevlex, etc, are
compatible with $\sigma$. For the algebra $P$ endowed with the map $\sigma$
we can clearly define the weight function $\w(x_i(j)) = j$. In Section 6
we show how to embed the free associative algebra $\KX$ into the skew polynomial
ring defined by $P$ and the monoid $\Sigma = \langle \sigma \rangle$.
Moreover, if we put $x_i(j) = \sigma^j(u_i)$ where $x_i(0) = u_i = u_i(t)$ is
a set of (algebraically independent) univariate functions and $\sigma$ is
the shift operator $u_i(t)\mapsto u_i(t+h)$ then $P = K[X\times\N]$ is
by definition the {\em ring of ordinary difference polynomials} with constant
coefficients in the field $K$ (see \cite{Le}). Such algebra is used
to study systems of (ordinary) difference equations for applications
in combinatorics or discretization of systems of differential
equations.

\begin{definition}
Let $I$ be an ideal of $P$. We call $I$ {\em $\w$-graded} if $I = \sum_i I_i$
with $I_i = I\cap P_i$ for any $i\in\hN$.
\end{definition}

Define now the skew monoid ring $\hS = P * \hSigma$ extending $S = P * \Sigma$
and let $\hpi:\hS\to P$ the left $\hS$-module epimorphism extending $\pi$
that is $s^i\mapsto 1$, for all $i\in\hN$. The existence of a weight function
for $P$ implies that one has also a mapping $\xi:P\to\hS$ such that
$\hpi \xi = id$.

\begin{proposition}
Define $\xi:P\to \hS$ the homogeneous injective map such that 
$f\mapsto \sum_{i\in\hN} f_i s^i$, for all $f = \sum_{i\in\hN} f_i\in P$.
Then $\xi$ is a $\hSigma$-equivariant map.
\end{proposition}

\begin{proof}
For all $i,j\in\hN$ and $f\in P_j$ one has that $s^i\cdot f\in P_{i+j}$
and therefore $\xi(s^i\cdot f) = (s^i\cdot f) s^{i+j} = s^i f s^j = s^i \xi(f)$.
\end{proof}

Let $I\subset P$ be a $\w$-graded $\Sigma$-invariant ideal and consider
$\xi(I)\subset\hS$. Note that if $I\neq P$ then $I_{-\infty} = 0$ and
the set $\xi(I)$ is in fact contained in $S$. Then, to get rid of the symbol
$-\infty$ we restrict ourselves to ideals not containing constants.

\begin{definition}
Let $I\subsetneq P$ be a $\w$-graded $\Sigma$-invariant ideal of $P$.
Denote by $I^S$ the graded (two-sided) ideal of $S$ generated by
$\xi(I)\subset S$. In other words, if we put $G = \xi(\bigcup_{i\geq 0} I_i) =
\{f s^i \mid f\in I_i, i\geq 0\}$ then $I^S$ is the left ideal
generated by $G\,\Sigma = \{f s^j \mid f\in I_i, j\geq i\geq 0\}$
or equivalently $I^S$ has the basis $G\,\Sigma = \Sigma\,G\,\Sigma$ as
$P$-submodule of $S$. We call $I^S$ the {\em skew analogue} of $I$. 
\end{definition}

\begin{proposition}
\label{S2Pcor}
Let $I\subsetneq P$ be a $\w$-graded $\Sigma$-invariant ideal.
Then $I^{SP} = I$, that is there is a {\em bijective correspondence}
between all $\w$-graded $\Sigma$-invariant ideals different from $P$ and
their skew analogues in $S$.
\end{proposition}

\begin{proof}
Put $J = I^{SP} = \pi(I^S)$. For any $f\in I_i$ and $j\geq i$
we have clearly $\pi(f s^j) = f$. Since the elements $f s^j$ are
a left basis of $I^S$, the ideal $I$ is $\Sigma$-invariant and $\pi$
is a left $S$-module homomorphism, we have that $J\subset I$. 
Moreover, the elements $f\in I_i$ are a basis of $I = \sum_i I_i$
and one has also that $I\subset J$.
\end{proof}

The next propositions need the following lemmas.

\begin{lemma}
\label{divbound}
If $s^k\cdot m$ divides $n$, with $m,n\in M$, then $\w(n) - k\geq \w(m)$.
\end{lemma}

\begin{proof}
Since $n = q (s^k\cdot m)$ with $q\in M$, we have $\w(n)\geq
\w(s^k\cdot m) = k + \w(m)$.
\end{proof}

\begin{lemma}
\label{lcmbound}
Let $m,n\in M$ and put $l = \lcm(m,n)$. Then, one has that $\w(l) =
\max(\w(m),\w(n))$.
\end{lemma}

\begin{proof}
By property (ii) of Definition \ref{weight}, it is sufficient to note that
$\max$ is an idempotent operation and hence the weight of a monomial
depends only on the variables occurring in the support.
\end{proof}

\begin{proposition}
Let $I\subsetneq P$ be a $\w$-graded $\Sigma$-invariant ideal,
then $I^S$ is a graded ideal of $S$. Let $G = \bigcup_i G_i$
be a {\em $\w$-homogeneous} $\Sigma$-basis of $I$ that is
$G_i\subset I_i$. Then $G^S = \xi(G) = \{f s^i \mid f\in G_i,i\geq 0\}$
is an $s$-homogeneous basis of $I^S$.
\end{proposition}

\begin{proof}
Consider the elements $f s^j$ with $f\in I_i,j\geq i$ which form
a left basis of $I^S$. Since $G$ is a $\Sigma$-basis, one has
$f = \sum_k f_k (s^k\cdot g_k)$ with $f_k\in P, g_k\in G_{i_k}$.
From $\w(f) = i$, by Lemma \ref{divbound} we obtain that
$i - k\geq i_k$. We have therefore that $f s^j =
\sum_k f_k (s^k\cdot g_k) s^k s^{j-k} =
\sum_k f_k s^k (g_k s^{j-k})$ with $j - k\geq i - k\geq i_k$
and hence $g_k s^{j-k}\in G^S\,\Sigma$, for all $k$.
\end{proof}

Note now that by Proposition \ref{toS2P} we have that $m s^i\prec n s^i$
if and only if $m\prec n$, for all $m,n\in M$ and for any $i\geq 0$.
In other words, if $f s^i$ ($f\in P$) is an $s$-homogeneous element
of $S$ then $\lm(f s^i) = \lm(f) s^i$.

\begin{lemma}
\label{gbgenIS}
Let $I\subsetneq P$ be a $\w$-graded $\Sigma$-invariant ideal.
If $G = \bigcup_i I_i$, by definition $I^S$ is the graded ideal
of $S$ generated by $G^S = \xi(G)$. Then $G^S$ is an $s$-homogeneous
\Gr\ basis of $I^S$.
\end{lemma}

\begin{proof}
Let $f s^i, g s^j\in G^S\,\Sigma$ that is the $\w$-homogeneous elements
$f,g\in G$ are such that $i\geq \w(f), j\geq \w(g)$. Assume $i\geq j$
and put $k = i - j$. By Proposition \ref{leftsigmacrit} we have to check for
\Gr\ representations of the S-polynomial $\spoly(f s^i, s^k g s^j)$
with respect to $\Sigma\,G^S\,\Sigma$. Since $G$ is clearly a \Gr\
$\Sigma$-basis of $I$, one has that the S-polynomial  $\spoly(f, s^k\cdot g)$
has a \Gr\ representation with respect to $\Sigma\cdot G$, say $h =
\spoly(f, s^k\cdot g) = \sum_l f_l (s^l\cdot g_l)$ with $f_l\in P,g_l\in G$.
Note that $\spoly(f s^i, s^k g s^j) = h s^i = \sum_l f_l (s^l\cdot g_l) s^i$.
We have to prove now that $i\geq l + \w(g_l)$ for any $l$, because
in this case one has the \Gr\ representation $h s^i =
\sum_l f_l s^l (g_l s^{i-l})$. In fact, by Lemma \ref{divbound}
and Lemma \ref{lcmbound} we have that $\max(\w(f),\w(g)) =
\w(h)\geq l + \w(g_l)$. Then, from $i\geq \w(f)$ and $i\geq j\geq \w(g)$
one obtains the claim.
\end{proof}

\begin{proposition}
Let $G\subset \bigcup_{i\geq 0} P_i$. Then $\lm(G)^S = \lm(G^S)$. Moreover,
if $I\subsetneq P$ is a $\w$-graded $\Sigma$-invariant ideal then
$\LM_P(I)^S = \LM(I^S)$.
\end{proposition}

\begin{proof}
If $f\in P_i$ is a $\w$-homogeneous element then $\w(\lm(f)) = \w(f) = i$ and
$\lm(f) s^i = \lm(f s^i)$. We obtain that $\lm(G)^S = \lm(G^S)$. Consider now
$G = \bigcup_i I_i$. By definition $I^S$ is the ideal of $S$ generated
by $G^S$. Moreover, since $I = \sum_i I_i$ one has that $\lm(G) = \lm(I)$
and hence $\LM_P(I)^S$ is the ideal generated by $\lm(G)^S = \lm(G^S)$.
Finally, by Lemma \ref{gbgenIS} one has that $\LM(I^S)$ is the ideal
of $S$ generated by $\lm(G^S)$.
\end{proof}

\begin{proposition}
\label{gbP2S}
Let $I\subsetneq P$ be a $\w$-graded $\Sigma$-invariant ideal.
Let $G = \bigcup_i G_i$ be a $\w$-homogeneous \Gr\ $\Sigma$-basis of $I$.
Then, $G^S = \xi(G)$ is an $s$-homogeneous \Gr\ basis of $I^S$.
\end{proposition}

\begin{proof}
By hypothesis $\lm(G)$ is a $\Sigma$-basis of $\LM_P(I)$.
Then $\lm(G^S) = \lm(G)^S$ is a basis of $\LM_P(I)^S = \LM(I^S)$
that is $G^S$ is a \Gr\ basis of $I^S$.
\end{proof}

\begin{proposition}
\label{gbS2P}
Let $I\subsetneq P$ be a $\w$-graded $\Sigma$-invariant ideal.
If $G$ is an $s$-homogeneous \Gr\ basis of $I^S$ then $G^P = \pi(G)$
is a \Gr\ $\Sigma$-basis of $I$.
\end{proposition}

\begin{proof}
Let $f\in I_l$ for some $l\geq 0$ and consider the element
$f s^l\in I^S$. Since $G$ is an $s$-homogeneous \Gr\ basis of $I^S$,
there is $g s^k\in G$ ($g\in P,k\geq 0$) such that
$\lm(f s^l) = q s^i \lm(g s^k) s^j$ that is $\lm(f) s^l = q s^i \lm(g) s^{k+j} =
q (s^i\cdot\lm(g)) s^l$ with $q\in M$ and $i + j + k = l$. It follows that
$\lm(f) = q (s^i\cdot\lm(g)) = q \lm(s^i\cdot g)$ with $g = \pi(g s^k)\in G^P$
and we conclude that $G^P$ is a \Gr\ $\Sigma$-basis of $I$.
\end{proof}

Note that Proposition \ref{gbP2S} and Proposition \ref{gbS2P} explain that
there is a complete equivalence between \Gr\ bases computations for
$\w$-graded $\Sigma$-invariant ideals $I\subsetneq P$ and their skew
analogues $I^S$ which are graded two-sided ideals of $S$. In particular,
\Gr\ $\Sigma$-bases of $I$ can be computed by the algorithm \SkewGBasis\
when applied to $I^S$. Precisely, if $H = \bigcup_i H_i$
is a $\w$-homogeneous $\Sigma$-basis of $I$ and $G = \SkewGBasis(H^S)$ then
$G^P = \pi(G)$ is a $\w$-homogeneous \Gr\ $\Sigma$-basis of $I$. 
We may call this procedure \SigmaGBasis2.

The following result provides algorithmic solution of the membership
problem for a class of $\Sigma$-invariant ideals. Note that such kind
of results are quite rare, for instance, in the theory of difference
ideals.

\begin{proposition}
\label{Ptermin}
Let $I\subset P$ be a $\w$-graded $\Sigma$-invariant ideal and fix $d\geq 0$.
Assume that $I$ has a $\w$-homogeneous basis $H$ such that $H_d =
\{f\in H\mid \w(f)\leq d\}$ is a finite set. Then, there is
a $\w$-homogeneous \Gr\ $\Sigma$-basis $G$ of $I$ such that $G_d$
is also a finite set. In other words, if we consider for the algorithm
\SigmaGBasis\ a selection strategy of the S-polynomials based on their
weights, we obtain that the $d$-truncated version of \SigmaGBasis\ stops
in a finite number of steps.
\end{proposition}

\begin{proof}
First of all, note that the algorithm \SigmaGBasis\ essentially
computes a subset $G$ of a \Gr\ basis $\Sigma\cdot G$ obtained
by applying the Buchberger algorithm to the basis $\Sigma\cdot H$
of $I$. Moreover, by property (iii) of Definition \ref{weight} and
Lemma \ref{lcmbound} the elements of $\Sigma\cdot H$ and hence of
$\Sigma\cdot G$ are all $\w$-homogeneous. Denote $H'_d =
\{ s^i\cdot f\mid i\geq 0, f\in H_d, i + \w(f)\leq d\}$. Since $H'_d$
is also a finite set, consider $X_d$ the finite set of variables
of $P$ occurring in the elements of $H'_d$ and define $P^{(d)} = K[X_d]$.
In fact, the $d$-truncated algorithm \SigmaGBasis\ computes a subset
of a \Gr\ basis of the ideal $I^{(d)}$ of $P^{(d)}$ generated
by $H'_d$. By Noetherianity of the ring $P^{(d)}$, we clearly obtain
termination.
\end{proof}

Note that the above result can be obtained also by Proposition \ref{termin}.
In fact, if $I\subsetneq P$ is finitely $\Sigma$-generated up to weight $d$
then $I^S$ is a graded ideal of $S$ which is finitely generated up to
$s$-degree $d$. Precisely, if $H = \bigcup_i H_i$ is a $\w$-homogeneous
$\Sigma$-basis of $I$ and the set $\bigcup_{i\leq d} H_i$ is finite
for all $d$, then $\{ f s^i \mid f\in H_i, i\leq d \}$ is also a finite
set that generates $I^S$ up to degree $d$.


\section{The skew letterplace embedding}

Denote $\N^* = \{1,2,\ldots\}$ the set of positive integers and
let $X = \{x_1,x_2,\ldots\}$ be a finite or countable set of variables.
We denote by $x_i(j)$ each element $(x_i,j)$ of the product set
$X\times\N^*$ and define $P = K[X\times\N^*]$ the polynomial ring in the
commuting variables $x_i(j)$. Consider the algebra monomorphism
of infinite order $\sigma:P \to P$ such that $x_i(j)\mapsto x_i(j+1)$
for all $i,j$. Note that $\sigma$ is a monomial map that is compatible
with divisibility in $\Mon(P)$. Then, put $S = P[s;\sigma]$ the skew polynomial
ring in the variable $s$ defined by $P$ and $\sigma$. Finally, let
$F = \KX$ denote the free associative algebra generated by $X$.
We consider $F$ as a graded algebra with respect to the total degree.
Recall that $S = \bigoplus_{i\in\N} S_i$ is also a graded algebra
with $S_i = P s^i$.

\begin{definition}
Let $A\subset S$ be a $K$-subalgebra. If $A$ is spanned by a submonoid
$M\subset\Mon(S)$ then we call $A$ {\em a monomial subalgebra of $S$}
and we denote $\Mon(A) = M$. In this case, a monomial ordering of $S$ can
be restricted to $A$.
\end{definition}

For instance, $P$ is a monomial subalgebra of $S$. We have now a result
about the possibility to embed the free associative algebra $F$
into the skew polynomial ring $S$.

\begin{proposition}
The graded algebra homomorphism $\iota:F\to S, x_i\mapsto x_i(1)s$
is injective. Then, the free associative algebra $F$ is isomorphic
to $R = \Im\iota$, a graded monomial subalgebra of $S$.
\end{proposition}

\begin{proof}
It is sufficient to note that by the commutation rule of the variable $s$
and the definition of the endomorphism $\sigma$, any word
$x_{i_1}\cdots x_{i_d}\in\Mon(F)$ maps into
$x_{i_1}(1)\cdots x_{i_d}(d)s^d\in\Mon(S)$.
\end{proof}

We call $S$ the {\em skew letterplace algebra} and the algebra monomorphism
$\iota$ the {\em skew letterplace embedding}. In Section 7 we will give
motivation for such names.
Fix now a monomial ordering $\prec$ on the algebra $S$ that is $\sigma$ is
compatible with the restriction of $\prec$ to $\Mon(P)$. It is easy to
show that many usual monomial orderings on $P$ (lex, degrevlex, etc)
satisfy such condition. Recall that the existence of monomial orderings
for $P$ is provided by the Higman's lemma which implies the following result
(see for instance \cite{AH}, Corollary 2.3 and remarks at beginning
of page 5175).

\begin{proposition}
Let $\prec$ be a total ordering on the set $\Mon(P)$ such that for all
$m,n,t\in \Mon(P)$ one has $1\preceq m$ and if $m\prec n$ then $t m \prec t n$.
Then $\prec$ is also a well-ordering of $\Mon(P)$ that is a monomial
ordering of $P$ if and only if the restriction of $\prec$ to the variables
set $X\times\N^*$ is a well-ordering.
\end{proposition}

Clearly, it is easy to assign well-orderings to the set $X\times\N^*$
which is in bijective correspondence to $\N^2$. Note that the algebra $P$
has also a multigrading which is defined as follows.
If $m = x_{i_1}(j_1)\cdots x_{i_d}(j_d)\in\Mon(P)$, then we denote
$\partial(m) = \mu = (\mu_k)_{k\in\N^*}$ where $\mu_k =
\#\{\alpha \mid j_\alpha = k\}$. If $P_\mu\subset P$ is the subspace
spanned by all monomials of multidegree $\mu$ then $P = \bigoplus_\mu P_\mu$
is clearly a multigrading of the algebra $P$. If $\mu = (\mu_k)$ is
a multidegree, we denote $i\cdot\mu = (\mu_{k-i})_{k\in\N^*}$ where
we put $\mu_{k-i} = 0$ when $k-i < 1$. By definition of the map $\sigma$,
if we denote $S_{\mu,i} = P_\mu s^i$ one obtains that
$S = \bigoplus_{\mu,i} S_{\mu,i}$ and $S_{\mu,i} S_{\nu,j}\subset
S_{\mu + (i\cdot\nu), i + j}$.
The elements of each subspace $S_{\mu,i}\subset S$ are said
{\em multi-homogeneous}. An ideal $J\subset S$ is called {\em multigraded}
if $J = \sum_{\mu,i} J_{\mu,i}$ with $J_{\mu,i} = J\cap S_{\mu,i}$.
In other words, the ideal $J$ is generated by multi-homogeneous elements.
For any integer $i\geq 0$ we denote by $1^i$ the multidegree $\mu =
(\mu_k)_{k\in\N^*}$ such that $\mu_k = 1$ if $k\leq i$ and $\mu_k = 0$
otherwise. Clearly, a homogeneous element $f s^i\in S$ ($f\in P$)
belongs to the graded subalgebra $R$ if and only if $f$ is multi-homogeneous
and $\partial(f) = 1^i$. In other words, $R_i = R\cap S_i = S_{1^i,i} =
P_{1^i} s^i$.

\begin{lemma}
\label{Rgood}
Let $f s^l\in S$ with $f\in P$ a multi-homogeneous element and consider
$f_{ij} s^i, g_j s^j, h_{jk} s^k\in S$ where $f_{ij}, g_j, h_{jk}\in P$ 
are multi-homogeneous elements such that $f s^l =
\sum_{i+j+k=l} f_{ij} s^i g_j s^j h_{jk} s^k$. Then, from $f s^l\in R$
it follows that $f_{ij} s^i, g_j s^j, h_{jk} s^k\in R$, for all $i,j,k$.
\end{lemma}

\begin{proof}
Clearly we have $f = \sum_{i+j+k=l} f_{ij} g_j^{s^i} h_{jk}^{s^{i+j}}$.
Denote $\mu = \partial(f_{ij}), \nu = \partial(g_j^{s^i})$ and
$\rho = \partial(h_{jk}^{s^{i+j}})$ and put $\alpha = \min\{k\mid\nu_k > 0\}$
and $\beta = \min\{k\mid\rho_k > 0\}$. By definition of the map $\sigma$,
one has that $\alpha\geq i+1$ and $\beta\geq i+j+1$. If we assume $f s^l\in R$
that is $1^l = \partial(f) = \mu + \nu + \rho$, then necessarily 
$\mu = 1^i, \nu = i\cdot 1^j$ and $\rho = (i+j)\cdot 1^k$ and hence
$\partial(f_{ij}) = 1^i, \partial(g_j) = 1^j, \partial(h_{jk}) = 1^k$.
\end{proof}

\begin{proposition}
\label{gengood}
Let $I$ be a graded (two-sided) ideal of $R\subset S$ and let $J$ be
the {\em extension of $I$ to $S$} that is $J$ is the (multigraded)
ideal generated by $I$ in $S$. If $G$ is a multi-homogeneous
basis of $J$ then $G\cap R$ is a (homogeneous) basis of $I$.
In particular, the {\em contraction} $J\cap R$ is equal to $I$,
that is there is a {\em bijective correspondence} between all graded
ideals of $R$ and their extensions to $S$.
\end{proposition}

\begin{proof}
Consider $f s^l\in I\subset R$ ($f\in P$) a homogeneous element and
let $G = \{g_j s^j\}$ with $g_j\in P$, $g_j$ multi-homogeneous.
Since $f$ is multi-homogeneous and $G$ is a basis of $J\supset I$,
one has $f s^l = \sum_{i+j+k=l} f_{ij} s^i g_j s^j h_{jk} s^k$
with $f_{ij},h_{jk}\in P$, $f_{ij},h_{jk}$ multi-homogeneous.
From Lemma \ref{Rgood} it follows immediately that all
elements $f_{ij} s^i, g_j s^j, h_{jk} s^k\in R$ that is $G\cap R$
is a basis of $I$.
\end{proof}

\begin{proposition}
\label{gbgood}
Let $I\subset R$ be a graded ideal and let $J\subset S$ be its
extension. If $G\subset J$ is a multi-homogeneous \Gr\ basis of $J$ then
$G\cap R$ is a homogeneous \Gr\ basis of $I$.  
\end{proposition}

\begin{proof}
If $f s^l = \sum_{i+j+k=l} f_{ij} s^i g_j s^j h_{jk} s^k$
is a \Gr\ representation in $S$ of a homogeneous element $f s^l\in I\subset J$
with respect to $G = \{g_j s^j\}$, then it is sufficient to use the same
argument of Proposition \ref{gengood} to obtain that $f s^l$ has a \Gr\
representation in $R$ with respect to $G\cap R$.
\end{proof}

We obtain finally an algorithm to compute \Gr\ bases of graded two-sided
ideals of the subring $R\subset S$ which is isomorphic to the free associative
algebra $F$ by the map $\iota$. Note that the considered monomial orderings
on $F$ are obtained as the restriction of monomial orderings on $S$ to the
monomial subalgebra $R$.
By applying Proposition \ref{gbgood}, the computation of homogeneous
\Gr\ bases in $R$ is obtained as a slight modification of the algorithm
\SkewGBasis\ for the ideals of $S$. It is interesting to note that
the latter procedure is in turn a variant of the Buchberger algorithm
for modules over commutative polynomial rings. Thus, we may say that
these computations in associative algebras are reduced to analogue ones
over commutative rings via the notion of skew polynomial ring (see also
Section 7). This reverses somehow the trivial fact that commutative algebras
are just a subclass of the associative ones. 

\suppressfloats[b]
\begin{algorithm}\caption{\FreeGBasis2}
\begin{algorithmic}[0]
\State \text{Input:} $H$, a homogeneous basis of a graded two-sided
ideal $I\subset R$.
\State \text{Output:} $G$, a homogeneous \Gr\ basis of $I$.
\State $G:= H$;
\State $B:= \{(f,g) \mid f,g\in G\}$;
\While{$B\neq\emptyset$}
\State choose $(f,g)\in B$;
\State $B:= B\setminus \{(f,g)\}$;
\ForAll{$i,j\geq 0, i + j = \deg_s(f) - \deg_s(g)$ and
$\spoly(f,s^i g s^j)\in R$}
\State $h:= \Reduce(\spoly(f,s^i g s^j), \Sigma\,G\,\Sigma)$;
\If{$h\neq 0$}
\State $B:= B\cup\{(h,h),(h,k),(k,h)\mid k\in G\}$;
\State $G:= G\cup\{h\}$;
\EndIf;
\EndFor;
\ForAll{$i,j\geq 0, j - i = \deg_s(f) - \deg_s(g)$ and
$\spoly(f s^i, s^j g)\in R$}
\State $h:= \Reduce(\spoly(f s^i, s^j g), \Sigma\,G\,\Sigma)$;
\If{$h\neq 0$}
\State $B:= B\cup\{(h,h),(h,k),(k,h)\mid k\in G\}$;
\State $G:= G\cup\{h\}$;
\EndIf;
\EndFor;
\EndWhile;
\State \Return $G$.
\end{algorithmic}
\end{algorithm}

\newpage

Note explicitely that conditions $\spoly(f,s^i g s^j), \spoly(f s^i, s^j g)\in R$
are equivalent to ask that such multi-homogeneous elements of $S$ have
multidegrees of type $(1^d,d)$, for some $d\geq 0$.

\begin{proposition}
The algorithm \FreeGBasis2\ is correct.
\end{proposition}

\begin{proof}
Since $G$ is multi-homogeneous implies that $\Sigma\,G\,\Sigma$ is also
multi-homogeneous, the procedure \Reduce\ clearly preserves multi-homogeneity.
Moreover, any element $f\in G$ ($f\notin H$) is obtained by reduction
of a S-polynomial, say $h$. Owing to Proposition \ref{gbgood} we are
interested only in the elements $f\in R$ and this holds if and only if
$h\in R$.
\end{proof}

Assume now that the graded ideal $I\subset R$ has a finite number of generators
up to some degree $d > 0$. Note that the $d$-truncated algorithm \FreeGBasis2\
has termination provided by termination of \SkewGBasis\ as stated in
Proposition \ref{termin}. This generalizes a well-known result about
algorithmic solution of the word problem (membership problem)
for finitely presented graded associative algebras.

\section{Letterplace in $P$}

As in Section 5, consider the $P$-linear map $\pi:S\to P$ such that
$s^i \mapsto 1$, for all $i$. Note now that $\iota' = \pi\iota:F\to P$
is an injective $K$-linear map such that
$x_{i_1}\cdots x_{i_d}\in\Mon(F)\mapsto x_{i_1}(1)\cdots x_{i_d}(d)\in\Mon(P)$. 
Recall that $F = \bigoplus_i F_i$ is a graded algebra with respect to
total degree. Moreover, consider the weight map $\w:\Mon(P)\to\hN$ such that
$\w(x_i(j)) = j$ for all $i,j\geq 1$ and the corresponding grading
$P = \bigoplus_{i\in\hN} P_i$ defined by the monoid $(\hN,\max)$.
Then, we have that $\iota'$ is a homogeneous map (note $K = F_0 = P_{-\infty}$
and $P_0 = 0$) and $\iota = \xi \iota'$ which is an algebra homomorphism.

\begin{definition}
Let $I\subsetneq F$ be a graded (two-sided) ideal. Denote by $I'\subsetneq P$
the $\w$-graded $\Sigma$-invariant ideal $\Sigma$-generated by $\iota'(I)$.
In other words, if $G = \{\iota'(f) \mid f\in I_i, i > 0\}$ then
$I'$ is the ideal of $P$ generated by $\Sigma\cdot G$. We call $I'$
the {\em letterplace analogue} of $I$.
\end{definition}

\begin{proposition}
Let $I\subsetneq F$ be a graded ideal and $I'\subsetneq P$ its letterplace
analogue. Denote by $J = I'^S$ the skew analogue of $I'$ and call $J$
the {\em skew letterplace analogue} of $I$. We have that $J$ is the extension
to $S$ of the ideal $\iota(I)\subset R$. Then, there is a {\em bijective
correspondence} between all graded ideals of $F$ and their (skew) letterplace
analogues.
\end{proposition}

\begin{proof}
Let $J'$ be the extension of $\iota(I)$ to $S$. By definition $J'$ is the ideal
generated by the elements $\iota(f) = \iota'(f) s^i$, for all $f\in I_i$.
Since $I'$ is $\Sigma$-generated by the $\w$-homogeneous elements $\iota'(f)$
of weight $i$, we conclude that $J = I'^S = J'$. Moreover, the bijective
correspondence between graded two-sided ideals of $F$ and their letterplace
analogues in $P$ is obtained by composing the bijections contained
in Proposition \ref{S2Pcor} and Proposition \ref{gengood}.
\end{proof}

The bijection between graded ideals of $F$ and their letterplace analogues
has been introduced in \cite{LSL} and called ``letterplace correspondence''.
The motivation of such name is essentially historical since the linear
map $\iota'$ was first considered in \cite{Fe,DRS}. Note that in these
articles the endomorphism $\sigma$ and the algebra embedding $\iota$
were not introduced. The polynomial ring $P$ was named there the ``letterplace
algebra'' because in the monomial $\iota'(x_{i_1}\cdots x_{i_d}) =
x_{i_1}(1)\cdots x_{i_d}(d)$ the indices $1,\ldots,d$ play the role
of the ``places'' where the ``letters'' $x_{i_1},\ldots,x_{i_d}$ occur
in the word $x_{i_1}\cdots x_{i_d}\in\Mon(F)$.

Fix now a monomial ordering $\prec$ on the algebra $S$ that is $\sigma$
is compatible with the restriction of $\prec$ to $\Mon(P)$. By restricting
$\prec$ to $R$ one obtains a monomial ordering on $F$. Denote by $V$
the image of the map $\iota'$ that is $V = \bigoplus_i V_i$ is a graded
subspace of $P$ where $V_i = P_{1^i}\subset P_i$. Note that $V$ is a left
$R$-module isomorphic to $R\approx F$. In fact, $V = \pi(R)$ and
the restriction $\pi:R\to V$ has the restriction $\xi:V\to R$ as its inverse.
In \cite{LSL} one has the following result which is now a direct
consequence of Proposition \ref{gbP2S} and Proposition \ref{gbgood}.

\begin{proposition}
Let $I\subsetneq F$ be a graded ideal and denote by $J\subsetneq P$ its
letterplace analogue. Then $J$ is a multigraded (hence $\w$-graded)
$\Sigma$-invariant ideal of $P$. If $G$ is a multi-homogeneous
(hence $\w$-homogeneous) \Gr\ $\Sigma$-basis of $J$ then $\iota'^{-1}(G\cap V)$
is a homogeneous \Gr\ basis of $I$.
\end{proposition}

From this result and algorithm \SigmaGBasis\ one obtains the correctness
of the following procedure which also has been introduced in \cite{LSL}.

\suppressfloats[b]
\begin{algorithm}\caption{\FreeGBasis}
\begin{algorithmic}[0]
\State \text{Input:} $H$, a homogeneous basis of a graded two-sided
ideal $I\subsetneq F$.
\State \text{Output:} $\iota'^{-1}(G)$, a homogeneous \Gr\ basis of $I$.
\State $G:= \iota'(H)$;
\State $B:= \{(f,g) \mid f,g\in G\}$;
\While{$B\neq\emptyset$}
\State choose $(f,g)\in B$;
\State $B:= B\setminus \{(f,g)\}$;
\ForAll{$i\geq 0$ such that $\spoly(f,s^i\cdot g)\in V$}
\State $h:= \Reduce(\spoly(f,s^i\cdot g), \Sigma\cdot G)$;
\If{$h\neq 0$}
\State $B:= B\cup\{(h,h),(h,k),(k,h),\mid k\in G\}$;
\State $G:= G\cup\{h\}$;
\EndIf;
\EndFor;
\EndWhile;
\State \Return $\iota'^{-1}(G)$.
\end{algorithmic}
\end{algorithm}

\newpage

Assume finally that the graded ideal $I\subsetneq F$ has a finite number
of generators up to some degree $d > 0$. Note that the $d$-truncated algorithm
\FreeGBasis\ has now termination provided by Proposition \ref{Ptermin}.

\section{Examples and timings}

In this section we propose an explicit computation and some timings
in order to provide some concrete experience with the algorithms
we introduced. 

Let $X = \{x\}$ and consider the ring of ordinary difference polynomials
$P = K[X\times\N]$ that is $P$ is the polynomial ring in the variables $x(j)$
which are the shifts of a single univariate function $x = x(0)$.
Moreover, let $P$ be endowed with the lexicographic monomial ordering
where $x(0) < x(1) < \ldots$. Denote by $J$ the difference ideal generated by
the single difference polynomial $g_1 = x(2)x(0) - x(1)$. This ideal
has been considered in \cite{GL} as an example of an ordinary difference
equation with periodic solutions. A variant of this equation has been also
considered in \cite{Ge}.
By the algorithm $\SigmaGBasis$ one can compute that $J$ has a finite
\Gr\ difference basis with elements
\[
\begin{array}{l}
g_1 = x(2)x(0) - x(1),
g_2 = x(4)x(1) - x(3)x(0),
g_3 = x(3)^2x(0) - x(3), \\
g_4 = x(4)x(3)x(0) - x(4),
g_5 = x(5) - x(4)x(0).
\end{array}
\]
Let us see how the algorithm \SigmaGBasis\ is able to obtain this.
If $\Sigma = \langle \sigma \rangle$ where $\sigma:P\to P$ is the
shift endomorphism $x(i)\mapsto x(i+1)$, then the ideal $J$ is
by definition $\Sigma$-generated by $\{g_1\}$ that is it is generated
by $\Sigma\cdot\{g_1\} = \{x(2)x(0) - x(1),x(3)x(1) - x(2),
x(4)x(2) - x(3),\ldots\}$. Up to the product criterion, to compute
a \Gr\ basis of $J$ one should consider all the S-polynomials
$\spoly(\sigma^i\cdot g_1,\sigma^{i+2}\cdot g_1)$ for any $i\geq 0$.
We are interested in fact in computing a \Gr\ $\Sigma$-basis of $J$ and
hence we can apply the $\Sigma$-criterion that kills all these S-polynomials
except for $\spoly(g_1,\sigma^2\cdot g_1)$. The reduction of this element
with respect to $\Sigma\cdot\{g_1\}$ leads to $g_2 = x(4)x(1) - x(3)x(0)$.
Now the current $\Sigma$-basis of $J$ is $\{g_1,g_2\}$. The S-polynomials that
survive to product and $\Sigma$-criterion are now
\[
\begin{array}{l}
\spoly(g_2,\sigma\cdot g_1),
\spoly(g_2,\sigma^2\cdot g_1),
\spoly(g_2,\sigma^4\cdot g_1),
\spoly(g_1,\sigma\cdot g_2), \\
\spoly(g_2,\sigma^3\cdot g_2).
\end{array}
\]
Then $\spoly(g_2,\sigma^2\cdot g_1)\to 0$ and $\spoly(g_2,\sigma\cdot g_1)$
reduces to $g_3$ with respect to $\Sigma\cdot\{g_1,g_2\}$. The list of new
S-polynomials arising from $g_3$ that pass product and $\Sigma$-criterion is
\[
\begin{array}{l}
\spoly(g_3,g_1),
\spoly(g_3,\sigma\cdot g_1),
\spoly(g_3,\sigma^3\cdot g_1),
\spoly(g_3,\sigma^2\cdot g_2), \\
\spoly(g_2,\sigma\cdot g_3),
\spoly(g_1,\sigma^2\cdot g_3),
\spoly(g_3,\sigma^3\cdot g_3),
\spoly(g_2,\sigma^4\cdot g_3).
\end{array}
\]
We have now that $\spoly(g_3,\sigma\cdot g_1)\to 0,
\spoly(g_3,g_1)\to 0$ and $\spoly(g_2,\sigma\cdot g_3)\to g_4$.
Up to all criteria, including chain criterion, the list of S-polynomials
has to be updated with the following ones
\[
\begin{array}{l}
\spoly(g_4,\sigma\cdot g_1),
\spoly(g_4,\sigma^2\cdot g_1),
\spoly(g_4,\sigma^3\cdot g_1),
\spoly(g_4,\sigma^4\cdot g_1), \\
\spoly(g_4,\sigma^3\cdot g_2),
\spoly(g_2,\sigma\cdot g_4),
\spoly(g_1,\sigma^2\cdot g_4),
\spoly(g_3,\sigma^3\cdot g_4), \\
\spoly(g_4,\sigma^3\cdot g_4),
\spoly(g_2,\sigma^4\cdot g_4),
\spoly(g_4,\sigma^4\cdot g_4).
\end{array}
\]
Now, one has the following reductions: $\spoly(g_4,\sigma\cdot g_1)\to 0,
\spoly(g_4,\sigma^2\cdot g_1)\to 0$ and $\spoly(g_1,\sigma\cdot g_2)\to
f = x(5)x(1) - x(3)x(0)^2$. We will show that the element $f$ of the
\Gr\ $\Sigma$-basis of $J$ is in fact redundant because $g_5$ is also
in this basis. The new S-polynomials arising from $f$ are
\[
\begin{array}{l}
\spoly(f,\sigma\cdot g_1),
\spoly(f,\sigma^5\cdot g_1),
\spoly(f,g_2),
\spoly(f,\sigma\cdot g_2),
\spoly(f,\sigma^4\cdot g_2), \\
\spoly(g_1,\sigma\cdot f),
\spoly(g_3,\sigma^2\cdot f),
\spoly(g_4,\sigma^2\cdot f),
\spoly(g_2,\sigma^3\cdot f), \\
\spoly(g_4,\sigma^3\cdot f),
\spoly(f,\sigma^4\cdot f).
\end{array}
\]
Then, we start again with reductions:
$\spoly(f,\sigma\cdot g_2)\to 0,\spoly(f,\sigma\cdot g_1)\to 0$,
$\spoly(g_3,\sigma^3\cdot g_1)\to 0,\spoly(g_2,\sigma\cdot g_4)\to g_5$
and therefore $f$ is redundant. The last S-polynomials to be added
are
\[
\begin{array}{l}
\spoly(g_5,\sigma^3\cdot g_1),
\spoly(g_5,\sigma^5\cdot g_1),
\spoly(g_5,\sigma\cdot g_2),
\spoly(g_5,\sigma^4\cdot g_2),\\
\spoly(g_5,f),
\spoly(g_5,\sigma^4\cdot f).
\end{array}
\]
If $G = \{g_1,g_2,g_3,g_4,g_5\}$ then we have that all remaining S-polynomials
to be considered reduce to zero with respect to $\Sigma\cdot G$, that is $G$
is a \Gr\ $\Sigma$-basis (difference basis) of the $\Sigma$-ideal
(difference ideal) $J$. 

We present now some timings obtained with an implementation of the algorithm
\FreeGBasis. This implementation, which is still under development,
is an improvement of the one we presented in \cite{LSL}. We decided not
to start implementing also \FreeGBasis2\ until \FreeGBasis\ will evolute
to some final form. We propose here new comparisons with the system
\textsc{Magma} that contains one of the most effective implementations
of the classical algorithm \cite{Mo,Gr1,Uf} for computing non-commutative
\Gr\ bases. Note that this implementation takes also advantage by the use of
Faug{\`e}re's F4 approach. The tests were performed on a PC with four Intel
Core i7 CPU 940 2.93GHz processors with 12 GB RAM running Ubuntu Linux.
We used \textsc{Singular} 3-1-3 with \texttt{freegb.lib} release 14203
and \textsc{Magma} version 2.17-8. We measured the time for real execution
of the process (thus differently to the way we did comparisons in \cite{LSL})
in "min:sec" format.  The number of generators in the input and in the output
are given as well.

\medskip
\begin{center}
\begin{tabular}[h]{|l|c|c|c|c|}
\hline
Example & {\sc Magma} & {\sc Singular} & $\#$In & $\#$Out \\
\hline
\texttt{G3-5-6-2d12} & 0:10 & 1:15 & 11 & 5885 \\
\texttt{G2-3-13-4d10} & 0:05 & 0:01 & 10 & 275 \\
\texttt{G3-8-13d8} & 0:05 & 0:04 & 18 & 1490 \\
\texttt{serf-g2d8} & 0:05 & 0:01  & 17 & 6 \\
\texttt{cliff5d9} & 0:08 & 0:12 & 41 & 168 \\
\texttt{C41d6}  & 0:05 & 0:10 & 6 & 50 \\
\texttt{C41Xd5} & 0:08 & 0:04 & 6 & 44 \\
\texttt{C41Yd5} & 0:05 & 0:03 & 6 & 44 \\
\texttt{C41Zd6} & 0:08 & 0:10 &  6 & 44 \\
\texttt{C41Wd6} & 0:05 & 0:01 & 6 & 35 \\
\hline
\end{tabular}
\end{center}

\medskip
This table shows essentially that the letterplace approach to the computation
of non-commutative \Gr\ bases is comparable with the classical algorithms
and hence it is feasible. From the viewpoint of implementations we record
that \textsc{Magma} achieved significant improvements with respect to comparisons
included in \cite{LSL} and this stimulate us to further optimize our code.
In fact, there is an ongoing work to enhance \texttt{freegb.lib} in
\textsc{Singular}. We will make more extensive comparisons in future articles
that will be concentrated on technical aspects of implementing the
letterplace algorithms.

Here is a brief description of the examples we considered for testing.
In all the examples the last integer indicates the total degree that bounds
the computations. The examples \texttt{G3-5-6-2, G2-3-13-4} refer to the class
of presented groups $G(l,m,n;q) = \langle r,s\mid r^l, s^m, (rs)^n,
[r,s]^q \rangle$, where $[r,s]$ denotes the commutator.
As for the example \texttt{G3-8-13}, this is one from the class of groups
$G(m,n,p) = \langle a,b,c\mid a^m, b^n, c^p, (ab)^2, (bc)^2, (ca)^2,
(abc)^2\rangle$. All these groups has been considered by Coxeter \cite{Co}
for the problem of determining their finiteness. For our computations
we considered a homogenization of the ideal of the free associative algebra
defining the group algebra of such groups.  The example \texttt{serf-g2}
are modified full Serre relations built from the Cartan matrix $G_2$.
The following non-commutative polynomials are explicitly the generators
we considered for homogenization.
\[
\begin{footnotesize}
\begin{array}{l}
f_1 f_2 f_2 - 2 f_2 f_1 f_2 + f_2 f_2 f_1,
e_1 e_2 e_2 - 2 e_2 e_1 e_2 + e_2 e_2 e_1,\\
f_1 f_1 f_1 f_1 f_2 - 4 f_1 f_1 f_1 f_2 f_1 + 6 f_1 f_1 f_2 f_1 f_1 -
4 f_1 f_2 f_1 f_1 f_1 + f_2 f_1 f_1 f_1 f_1, \\
e_1 e_1 e_1 e_1 e_2 - 4 e_1 e_1 e_1 e_2 e_1 + 6 e_1 e_1 e_2 e_1 e_1 -
4 e_1 e_2 e_1 e_1 e_1 + e_2 e_1 e_1 e_1 e_1,\\
f_2 e_1 - e_1 f_2,
f_1 e_2 - e_2 f_1,
f_1 e_1 - e_1 f_1 + h_1,
f_2 e_2 - e_2 f_2 + h_2,\\
h_1 h_2 - h_2 h_1,
h_1 e_1 - e_1 h_1 - 2 e_1,
f_1 h_1 - h_1 f_1 - 2 f_1,
h_1 e_2 - e_2 h_1 + e_2,\\
f_2 h_1 - h_1 f_2 + f_2,
h_2 e_1 - e_1 h_2 + 3 e_1,
f_1 h_2 - h_2 f_1 + 3 f_1,
h_2 e_2 - e_2 h_2 - 2 e_2,\\
f_2 h_2 - h_2 f_2 - 2 f_2.
\end{array}
\end{footnotesize}
\]
Let $F_5 = K\langle x_1,x_2,x_3,x_4,x_5 \rangle$ and define
$\Gamma\subset \End_K(F_5)$ the submonoid of all algebra endomorphisms
sending variables into variables. The example \texttt{cliff5} is the ideal
$I\subset F_5$ which is $\Gamma$-generated by the polynomials
$[x_1\circ x_2,x_3] = (x_1x_2 + x_2x_1) x_3 - x_3 (x_1x_2 + x_2x_1)$ and
$s_5 = \sum_{\pi\in\S_5} \sgn(\pi)
x_{\pi(1)}x_{\pi(2)}x_{\pi(3)}x_{\pi(4)}x_{\pi(5)}$. The quotient ring $F_5/I$
is the generic Clifford algebra in 5 variables of a 4-dimensional vector space.
Finally, the family \texttt{C41} of examples originates from random linear
substitutions into the ideal of 6 generators, defining the non-cancellative
monoid $C(4,1)$ (see \cite{JO}) and includes also variations of those.
For instance, \texttt{C41W} is given by
\[
\begin{array}{l}
x_4 x_4-25 x_4 x_2-x_1 x_4-6 x_1 x_3-9 x_1 x_2+x_1 x_1, \\
x_4 x_3+13 x_4 x_2+12 x_4 x_1-9 x_3 x_4+4 x_3 x_2+41 x_3 x_1-7 x_1 x_4-x_1 x_2, \\
x_3 x_3-9 x_3 x_2+2 x_1 x_4+x_1 x_1,
17 x_4 x_2-5 x_2 x_2-41 x_1 x_4, \\
x_2 x_2-13 x_2 x_1-4 x_1 x_3+2 x_1 x_2-x_1 x_1,
x_2 x_1+4 x_1 x_2-3 x_1 x_1.
\end{array}
\]
while \texttt{C41} is given by
\[
\begin{footnotesize}
\begin{array}{l}
189 x_4 x_4+63 x_4 x_3-66 x_4 x_2-161 x_4 x_1-103 x_3 x_4+19 x_3 x_3+262 x_3 x_2+467 x_3 x_1 -\\
360 x_2 x_4-144 x_2 x_3+24 x_2 x_2+136 x_2 x_1+175 x_1 x_4+35 x_1 x_3-160 x_1 x_2-315 x_1 x_1,\\

 27 x_4 x_4+409 x_4 x_3+82 x_4 x_2-42 x_4 x_1-57 x_3 x_4-403 x_3 x_3+26 x_3 x_2-42 x_3 x_1 -\\
50 x_2 x_4-434 x_2 x_3-12 x_2 x_2-14 x_2 x_1+45 x_1 x_4+435 x_1 x_3+30 x_1 x_2,\\

232 x_4 x_4-29 x_4 x_3+77 x_4 x_2+332 x_4 x_1-147 x_3 x_4+175 x_3 x_3+60 x_3 x_2-269 x_3 x_1 -\\
107 x_2 x_4+184 x_2 x_3+83 x_2 x_2-217 x_2 x_1+28 x_1 x_4-217 x_1 x_3-139 x_1 x_2+120 x_1 x_1,\\ 

52 x_4 x_4+233 x_4 x_3-129 x_4 x_2+135 x_4 x_1-248 x_3 x_4-205 x_3 x_3+171 x_3 x_2+138 x_3 x_1 +\\
100 x_2 x_4-58 x_2 x_3-177 x_2 x_1+84 x_1 x_4+39 x_1 x_3-43 x_1 x_2-73 x_1 x_1,\\

-225 x_4 x_4-150 x_4 x_3-179 x_4 x_2-262 x_4 x_1+91 x_3 x_4-94 x_3 x_3+225 x_3 x_2+74 x_3 x_1 +\\
214 x_2 x_4+224 x_2 x_3+90 x_2 x_2+266 x_2 x_1-175 x_1 x_4-50 x_1 x_3-205 x_1 x_2-190 x_1 x_1,\\

289 x_4 x_4-170 x_4 x_3-289 x_4 x_2-153 x_4 x_1-186 x_3 x_4+95 x_3 x_3+177 x_3 x_2+106 x_3 x_1 -\\
231 x_2 x_4+35 x_2 x_3+168 x_2 x_2+175 x_2 x_1+241 x_1 x_4+60 x_1 x_3-115 x_1 x_2-233 x_1 x_1.
\end{array}
\end{footnotesize}
\]


\section{Conclusions and future directions}

From the previous sections we can conclude that, owing to the notion
of \Gr\ $\Sigma$-basis and the skew letterplace embedding $\iota$,
the theory of non-commutative \Gr\ bases developed for the free associative
algebra $F = \KX$ using the concepts of overlappings, tips or obstructions
\cite{Gr1,Mo,Uf} can be deduced from, unified to the classical Buchberger theory
for commutative polynomial rings based on S-polynomials, at least
in the graded case. From a practical point of view, one obtains
the alternative algorithms \FreeGBasis\ and \FreeGBasis2\ which are
implementable in any computer algebra system providing
commutative \Gr\ bases. The feasibility of such methods has been already
shown in \cite{LSL} and confirmed by the new timings we have collected
in Section 8. 

Moreover, the general theory developed in this paper can be applied
to any context where a monoid of endomorphisms $\Sigma$ acts on
the polynomial algebra $P = K[X]$ in a way which is compatible with
\Gr\ bases theory. We propose not only an abstract definition of what
this may mean contributing to a current research trend (see for instance
\cite{DLS,AH,BD}), but also a method to transfer the related algorithms
from $P$ to the skew monoid ring $S = P *\Sigma$ when a suitable grading
is given for $P$. This theory applies in particular to the shift operators
and hence a stimulating field of applications are the rings of difference
polynomials. The simple calculation proposed in Section 8 gives some
feeling of this. In particular, we aim to extend the \Gr\ $\Sigma$-bases theory
to any finitely generated free commutative monoid $\Sigma =
\langle \sigma_1,\ldots,\sigma_r \rangle$ in order to cover partial difference
ideals and to extend the letterplace method for $F$ to the non-graded
case by means of suitable (de)homogenization techniques. An effective
implementation of all proposed algorithms will be clearly important
to understand the actual performance of the methods.

\section*{Acknowledgments}

First of all, we would like to thank Vladimir Gerdt for introducing us
to the theory of difference polynomial rings and thus leading us to note
that the letterplace algebra $K[X\times\N]$ endowed with the endomorphism
$x_i(j)\mapsto x_i(j+1)$ is just an instance of them (ordinary case).
We are grateful to Albert Heinle, Benjamin Schnitzler and Grischa Studzinski
for adopting the {\sc SymbolicData} project \cite{SD} for handling ideals
in free algebras. The timings in this article were obtained with the help
of the updated {\sc SymbolicData} system. We like also to thank
the anonymous reviewers for their valuable comments and suggestions.


\end{document}